\documentclass[11pt]{amsart}
\usepackage{geometry}
\usepackage{amscd,amssymb,verbatim,xcolor}
\usepackage{longtable, multirow}
\usepackage{soul,cancel,tikz}
\usetikzlibrary {arrows.meta,bending}
\date{\today}

\newcommand\fg{{\mathfrak g}}
\newcommand\fh{{\mathfrak h}}

\newcommand\fl{{\mathfrak l}}

\newcommand\fp{{\mathfrak p}}
\newcommand\fq{{\mathfrak q}}

\newcommand\fu{{\mathfrak u}}

\newtheorem{theorem}{Theorem}[section]
\newtheorem{algorithm}[theorem]{Algorithm}

\newtheorem{corollary}[theorem]{Corollary}
\newtheorem{conjecture}[theorem]{Conjecture}
\newtheorem{definition}[theorem]{Definition}
\newtheorem{example}[theorem]{Example}
\newtheorem{lemma}[theorem]{Lemma}
\newtheorem{proposition}[theorem]{Proposition}
\newtheorem{remark}[theorem]{Remark}

\begin{document}
\title[Unitary dual of $U(p,q)$]{On some conjectures of the unitary dual of $U(p,q)$}
\author{Kayue Daniel Wong}

\address{Mathematics Research Center, School of Science and Engineering, The Chinese University of Hong Kong, Shenzhen, Guangdong 518172, China }
\address{Institute of Mathematical Sciences, The Chinese University of Hong Kong, Shatin, Hong Kong, China}
\email{kayue.wong@gmail.com}

\begin{abstract}
In this manuscript, we introduce the notion of fundamental cases to study the unitary dual of $U(p,q)$. As applications, we prove of a conjecture of Salamanca-Riba and Vogan stated in 1998, as well as the fundamental parallelepiped (FPP) conjecture of Vogan in 2023 for $U(p,q)$. 
\end{abstract}

\maketitle
\setcounter{tocdepth}{1}

\section{Introduction}\label{sec:intro}
A main unsolved problem in the representation theory of real reductive Lie groups $G$ is the classification of all irreducible, unitarizable $(\mathfrak{g},K)$-modules, i.e. the unitary dual $\widehat{G}$. In \cite{SRV98}, Salamanca-Riba and Vogan suggested that one can begin by partitioning all irreducible, admissible 
$(\mathfrak{g},K)$-modules by their lowest $K$-types, and reduce the study of $\widehat{G}$ by focusing on $(\mathfrak{g},K)$-modules with {\bf unitarily small lowest $K$-types} (see Section \ref{sec-SV} for details).  More precisely, they proposed that the lowest $K$-types of all irreducible, unitary $(\fg,K)$-modules must either be unitarily small, or it is a component of a cohomologically induced module from a unitary representation in the good range.

\smallskip
In order for the reduction described above to work, they conjectured that in order for an irreducible $(\mathfrak{g}, K)$-module $X$ with real infinitesimal character $\Lambda$ and unitarily small lowest $K$-types to be unitary, a necessary condition on $\Lambda$ must be satisfied 
(see \cite[Conjecture 5.7]{SRV98} or Conjecture \ref{conj-original} below). By \cite[Theorem 5.8]{SRV98}, the validity of this conjecture immediately implies that one can study $\widehat{G}$ through the above reduction process. 

\smallskip
On the other hand, Vogan \cite{V23} recently proposed the 
{\bf Fundamental Parallelepiped (FPP) conjecture} for $G$. Namely, suppose $X$ is an irreducible $(\mathfrak{g}, K)$-module with real infinitesimal character $\Lambda$ such that it is not cohomologically induced from any proper $\theta$-stable parabolic subalgebras within the good range, then {\it another} necessary condition on $\Lambda$ needs to be satisfied in order for $X$ to be unitary (Conjecture \ref{conj-fpp}).

\smallskip

Essentially, both conjectures propose that under different hypotheses on $X$, there are different necessary conditions on its infinitesimal character so that $X$ is unitary.
The main objective of this manuscript is to view these conjectures from a unified standpoint, so that one can prove the following:
\begin{theorem} \label{thm-main}
Conjecture 5.7 of \cite{SRV98} and the FPP conjecture hold for $G = U(p,q)$.
\end{theorem}

More explicitly, the proof of Theorem \ref{thm-main} follow almost immediately on a detailed study of the {\bf fundamental cases}, whose precise definition is given in Section \ref{sec-fund}. These cases cover all the {\bf basic cases} defined in \cite{KS83}, which were effectively applied to study a large part of the unitary dual of $U(p,2)$. One therefore expects that the fundamental cases will play an vital role in the study of the unitary dual of $U(p,q)$ and, more generally, for all real reductive groups. 

\medskip
The manuscript is organized as follows: In Section \ref{sec-prelim}, we recall Langlands' classification of irreducible admissible $(\mathfrak{g},K)$-modules in \cite{V82}, and what it means for these modules to be unitary. Subsequently, we state the conjectures of Salamanca-Riba-Vogan (Conjecture \ref{conj-original}) and Vogan (Conjecture \ref{conj-fpp}). In Section \ref{sec-nonunit}, we give some general tools for detecting non-unitarity of irreducible $(\mathfrak{g},K)$-modules. Afterwards, we focus on the case of $G = U(p,q)$ in Section \ref{sec-upq}, and introduce some combinatorics relevant to the parametrization of irreducible modules of these groups. In Section \ref{sec-fund}, we define the notion of fundamental modules for $U(p,q)$, and prove a non-unitarity theorem (Theorem \ref{thm-upq}) for these modules. Finally, Section \ref{sec-general} and Section \ref{sec-fpp} are devoted to proving Salamanca-Riba-Vogan's conjecture and Vogan's FPP conjecture for $U(p,q)$ respectively.



\section{Preliminaries} \label{sec-prelim}
\subsection{Langlands classification of irreducible $(\mathfrak{g},K)$-modules}
In this section, we follow \cite{V82} for a construction of all irreducible $(\mathfrak{g},K)$-modules using \emph{$\theta$-stable data}.

Let $G$ be a connected reductive Lie group, with Cartan decomposition $\mathfrak{g}_0 = \mathfrak{k}_0+\mathfrak{p}_0$, that is, $\mathfrak{k}_0$ and $\mathfrak{p}_0$ are the $+1$ and $-1$ eigenspaces of a Cartan involution $\theta$ of $\mathfrak{g}_0$. We drop the subscript $0$ for the complexified Lie algebras.

Let $\mathfrak{h}_0 = \mathfrak{t}_0 + \mathfrak{a}_0$ be the fundamental Cartan subalgebra of $\mathfrak{g}_0$, that is, $\mathfrak{t}_0$ is the Cartan subalgebra of $\mathfrak{k}_0$ (up to conjugation). Fix root systems $\Delta(\mathfrak{k},\mathfrak{t})$ and $\Delta(\mathfrak{g},\mathfrak{t})$ such that
$\Delta(\mathfrak{k},\mathfrak{t}) \subseteq \Delta(\mathfrak{g},\mathfrak{t})$, and a positive root system $\Delta^+(\mathfrak{k},\mathfrak{t})$. 

By Cartan-Weyl's highest weight theory, every irreducible representation $\delta \in \widehat{K}$ has a highest weight $\mu(\delta)$, where $\mu(\delta)$ is a dominant weight in $\Delta^+(\mathfrak{k},\mathfrak{t})$.
For any admissible $(\mathfrak{g},K)$-module $X$, we say $\delta \in \widehat{K}$ is a {\bf $K$-type of $X$} if the multiplicity $m^X(\delta) := [\delta : X|_K]$ is nonzero. We say $\delta$ is a {\bf lowest $K$-type} of $X$ if the value
$\|\mu(\delta)+2\rho(\mathfrak{k})\|$ is the smallest among all $K$-types of $X$ (from now on, we write $\rho(\Phi)$ as half the sum of all positive roots in the set of roots $\Phi$, and $\rho(\mathfrak{w}) := \rho(\Delta^+(\mathfrak{w},\mathfrak{t}))$).

For any $K$-dominant weight $\mu \in \Delta^+(\mathfrak{k},\mathfrak{t})$, define
\begin{equation} \label{def-lambdaa}
\lambda_a(\mu) := P(\mu + 2\rho(\Delta^+(\mathfrak{k},\mathfrak{t})) - \rho(\Delta^+(\mathfrak{g},\mathfrak{t}))) \in \mathfrak{t}^*,
\end{equation}
where $\Delta^+(\mathfrak{g},\mathfrak{t})$ is chosen such that $\mu + 2\rho(\Delta^+(\mathfrak{k},\mathfrak{t}))$ is dominant, and $P$ is the projection map to the dominant Weyl chamber as defined in \cite[Definition 1.2]{SRV98}. We also write $\lambda_a(\delta) := \lambda(\mu(\delta))$ for any $\delta \in \widehat{K}$.

\smallskip
Also, For each $\lambda_a \in \mathfrak{t}^*$, define a $\theta$-stable parabolic subalgebra
$$\mathfrak{q}(\lambda_a) = \mathfrak{g}(\lambda_a) + \mathfrak{u}(\lambda_a)$$
satisfying $\Delta(\mathfrak{g}(\lambda_a),\mathfrak{t}) = \{\alpha \in \Delta(\mathfrak{g},\mathfrak{t})\ |\ \langle \alpha,\lambda_a\rangle = 0\}$ and $\Delta(\mathfrak{u}(\lambda_a),\mathfrak{t}) = \{\alpha \in \Delta(\mathfrak{g},\mathfrak{t})\ |\ \langle \alpha,\lambda_a\rangle > 0\}$.

The theorem below explains the importance of $\lambda_a$:
\begin{theorem}[\cite{SRV98} Theorem 2.9] \label{thm-lambdaa} 
Let $G$ be a connected reductive Lie group. 

\noindent (a) Consider
$$\Pi_a^{\lambda_a}(G) := \Big\{\begin{matrix} \pi\ \text{adm. irred.}\\ 
(\mathfrak{g},K)\text{-module}\end{matrix}\ \Big|\ 
\begin{matrix}\text{a lowest K-type}\ \delta\ \text{of}\ \pi\\ 
\text{satisfies}\ \lambda_a(\delta) = \lambda_a \end{matrix}\Big\}.$$ 
There is a bijection
\begin{equation} \label{eq-lambdaa}
\Pi_a^{\lambda_a - \rho(\mathfrak{u}(\lambda_a))}(G(\lambda_a)) \longrightarrow \Pi_a^{\lambda_a}(G)
\end{equation}
More precisely, let $Z_a \in \Pi_a^{\lambda_a - \rho(\mathfrak{u}(\lambda_a))}(G(\lambda_a))$ be an irreducible $(\mathfrak{g}(\lambda_a), G(\lambda_a) \cap K)$-module. Then the following holds:
\begin{itemize}
\item The lowest $G(\lambda_a) \cap K$-types of $Z_a$,
$$\{\eta_1, \dots,\eta_r\}$$
satisfy $\lambda_a(\eta_1) = \dots = \lambda_a(\eta_r) = \lambda_a - \rho(\mathfrak{u}(\lambda_a))$;
\item There is a unique irreducible subquotient $X$ in the cohomologically induced module $\mathcal{L}_{\mathfrak{q}(\lambda_a),S}(Z_a)$
(see \cite[Section 5]{KV95} for more details) for $S := \dim(\mathfrak{u}(\lambda_a) \cap \mathfrak{p})$, whose lowest $K$-types are precisely
$$\{\mathcal{L}_{\mathfrak{q}(\lambda_a),S}^K(\eta_1), \dots, \mathcal{L}_{\mathfrak{q}(\lambda_a),S}^K(\eta_r)\}$$
satisfying $\lambda_a(\mathcal{L}_{\mathfrak{q}(\lambda_a),S}^K(\eta_1)) = \dots = \lambda_a(\mathcal{L}_{\mathfrak{q}(\lambda_a),S}^K(\eta_r)) = \lambda_a$.
\end{itemize}
And the bijection \eqref{eq-lambdaa} is given by
$$Z_a \mapsto X.$$
\noindent (b) Suppose $\beta$ is a $(G(\lambda_a) \cap K)$-type in $Z_a$ such that $\mu(\beta) + 2\rho(\mathfrak{u}(\lambda_a) \cap \mathfrak{p})$ is $K$-dominant (this includes all lowest $K$-types $\eta_i$ of $Z_a$), then $\mathcal{L}_{\mathfrak{q}(\lambda_a),S}^K(\beta) \in \widehat{K}$ is irreducible with highest weight $\mu(\beta) + 2\rho(\mathfrak{u}(\lambda_a) \cap \mathfrak{p})$. Moreover, the map \eqref{eq-lambdaa} maps these $K$-types injectively to $X$, so that
$$m^{Z_a}(\beta) = m^X\left(\mathcal{L}_{\mathfrak{q}(\lambda_a),S}^K(\beta)\right)$$
These $\mathcal{L}_{\mathfrak{q}(\lambda_a),S}^K(\beta)$'s are called {\bf $\mathfrak{q}(\lambda_a)$-bottom layer $K$-types} of $X$.
\end{theorem}

In view of the above theorem, it is essential to classify the irreducible representations $Z_a\in \Pi_a^{\lambda_a - \rho(\mathfrak{u}(\lambda_a))}(G(\lambda_a))$. 
By \cite[Proposition 4.1]{SRV98}, $L := G(\lambda_a)$ is quasisplit for all $\lambda_a = \lambda_a(\delta)$, and $\lambda_a' := \lambda_a - \rho(\mathfrak{u}(\lambda_a))$ is central in $L$.

Let $B' = T'A'N'$ be a minimal parabolic subgroup of $L$, whose Levi subgroup $H' = T'A'$ is a maximally split Cartan subgroup of
$L$. Assume the identity component of $T'$ is contained in the maximal torus $T$ of $K$, such that 
$\lambda_a'$ vanishes on the orthogonal complement of $\mathfrak{t}' \subseteq \mathfrak{t}$,
and hence can be seen as an element $\lambda_a' \in (\mathfrak{t}')^*$. 
%
\begin{theorem}[\cite{SRV98} Proposition 4.1(c)] \label{thm-quasisplit}
Retain the above notations, and let $Z_a \in \Pi_a^{\lambda_a'}(L)$. Suppose $\eta$ is a lowest $(L \cap K)$-type
of $Z_a$, and $\gamma \in \widehat{T'}$ be a factor in the restriction $\eta|_{T'}$, then
there is $\nu \in (\mathfrak{a}')^*$ such that $Z_a$ is a Langlands subquotient
of the principal series 
\begin{equation} \label{eq-ps}
Ind^{L}_{T'A'N'}(\gamma \boxtimes \nu \boxtimes 1).
\end{equation}
In particular, the infinitesimal character of $Z_a$ is equal to $(\lambda_a',\nu) \in (\mathfrak{t}')^* +(\mathfrak{a}')^*$.
\end{theorem}
%
%
%
To conclude, we have the following algorithm constructing all irreducible $(\mathfrak{g},K)$-modules.
\begin{algorithm} \label{alg-langlands}
For any $\delta \in \widehat{K}$, apply the following algorithm:
\begin{itemize}
\item[(i)] Compute $\lambda_a := \lambda_a(\delta)$, and obtain the $\theta$-stable parabolic subalgebra $\mathfrak{q}(\lambda_a) = \mathfrak{g}(\lambda_a) + \mathfrak{u}(\lambda_a)$;
\item[(ii)] Let $L := N_G(\mathfrak{q}(\lambda_a))$ be a quasisplit Levi subgroup of $G$, and $H'= T'A'$ be a maximally split Cartan subgroup of $L$;
\item[(iii)] Let $\eta \in \widehat{L \cap K}$ has highest weight $\mu(\delta) - 2\rho(\mathfrak{u}(\lambda_a) \cap \mathfrak{p})$, so that $\mathcal{L}_{\mathfrak{q}(\lambda_a),S}^K(\eta) = \delta$;
\item[(iv)] Take {\bf any} $\gamma \in \widehat{T'}$ such that $[\gamma: \eta|_{T'}] > 0$; and {\bf any} $\nu \in (\mathfrak{a}')^*$.
\end{itemize}
Then there is a unique irreducible subquotient $X$ of 
\begin{equation} \label{eq-langlands}
\mathcal{L}_{\mathfrak{q},S}\left(Ind_{T'A'N'}^{L}(\gamma \boxtimes \nu \boxtimes 1)\right)
\end{equation}
with infinitesimal character $\Lambda := (\lambda_a, \nu)$ such that $X$ contains $\delta$ as a lowest $K$-type. 
Moreover, all irreducible modules with lowest $K$-type $\delta$ are obtained in such a way.
\end{algorithm}
The main theme in \cite[Chapter 4]{V82} is to understand the structure theory of the principal series \eqref{eq-ps} appearing in \eqref{eq-langlands}
for various choices of $\gamma$ and $\nu$. For instance, if $\delta' \neq \delta \in \widehat{K}$ satisfies $\lambda_a(\delta') = \lambda_a(\delta)$, 
the above algorithm may yield the same $\gamma$ (see Example \ref{eg-lkts} below). The notion of {\bf $R$-groups} is to determine whether 
$\delta$ and $\delta'$ belong to different irreducible subquotients of \eqref{eq-langlands}
for different values of $\nu$. We will study them in full detail for $G = U(p,q)$.

\subsection{Hermitian forms and unitarity}
It is well known that the classification of $\widehat{G}$ is equivalent to the classification of irreducible $(\mathfrak{g}, K)$-modules admitting positive definite $\mathfrak{g}$-invariant Hermitian forms. 
The following result determines a necessary and sufficient condition for an irreducible representation having an invariant, nondegenerate Hermitian form:
\begin{theorem}[\cite{KZ76}] \label{thm-herm}
Let $X$ be an irreducible admissible $(\mathfrak{g},K)$-module, which is a subquotient of the induced module \eqref{eq-langlands} via the construction in Algorithm \ref{alg-langlands}. In particular, $X$ has a lowest $K$-type $\delta$ and infinitesimal character
$\Lambda = (\lambda_a(\delta),\nu)$. Then $X$ has an invariant nondegenerate Hermitian form if and only if there exists $w \in W(G,H')$ such that  
$$w(\lambda_a(\delta),\nu) = (\lambda_a(\delta),-\overline{\nu}).$$
Moreover, all such invariant forms on $X$ are unique up to a non-zero scalar.
\end{theorem}
\begin{corollary}
The bijection map \eqref{eq-lambdaa} maps Hermitian modules in $\Pi_a^{\lambda_a - \rho(\mathfrak{u}(\lambda_a))}(G(\lambda_a))$ bijectively onto Hermitian modules in $\Pi_a^{\lambda_a}(G)$.
\end{corollary}
\begin{proof}
Let $\lambda_a' = \lambda_a - \rho(\mathfrak{u}(\lambda_a))$ and $L = G(\lambda_a)$ as before.
By \cite[Theorem 2.13]{SRV98}, the map \eqref{eq-lambdaa} maps Hermitian modules injectively to Hermitian modules. As for surjectiveness,
note that if $X$ is a Hermitian $(\mathfrak{g},K)$-module, then the above theorem says $w\lambda_a = \lambda_a$, i.e. $w \in W(L,H')$.

Now consider the $(\mathfrak{l}, L \cap K)$-module $Z_a$ which is the preimage of $X$ in \eqref{eq-lambdaa}. It has infinitesimal character $(\lambda_a', \nu) \in (\mathfrak{t'})^* + (\mathfrak{a'})^*$. So the same $w \in W(L,H')$ satisfies $w(\lambda_a',\nu) = (\lambda_a',-\overline{\nu})$, and hence $Z_a$ itself is also Hermitian.
\end{proof}

Recall that $\Lambda = (\lambda_a(\delta),\nu)$ is the infinitesimal character of $X$. By standard `reduction to real infinitesimal characters' argument (e.g. Theorem 2.6 of \cite{B04}), one can assume that $Im(\nu) = 0$, i.e. $\nu \in \mathfrak{a}_0^*$. We will assume $\Lambda$ has real infinitesimal character for the rest of this manuscript.

\begin{definition} 
Let $X$ be irreducible $(\mathfrak{g},K)$-module with an invariant Hermitian form $\langle\ ,\ \rangle_G$. The {\bf signature} 
$$(p^X(\delta),q^X(\delta))$$
of a $K$-type $\delta$ in $X$ is the signature of the induced Hermitian form
of the $m^X(\delta)$-dimensional vector space $Hom_K(\delta, X)$. In particular, 
$m^X(\delta) = p^X(\delta) + q^X(\delta)$ for all $\delta \in \widehat{K}$.
\end{definition}

By the last statement of Theorem \ref{thm-herm}, in order to check whether $X$ is unitary, 
it suffices to study \emph{one} Hermitian form of $X$,
and determine whether $p^X$, $q^X: \widehat{K} \to \mathbb{N}$ are zero functions or not.

\smallskip
Here is a refined statement for Theorem \ref{thm-lambdaa}(b):
\begin{theorem}[\cite{SRV98}, Proposition 2.10] \label{thm-bottomlayer}
Recall the bijection $Z_a \mapsto X$ in Theorem \ref{thm-lambdaa}. Suppose $Z_a$ has
a Hermitian form $\langle\ ,\ \rangle_L$, and $\langle\ ,\ \rangle_G$
is the Hermitian form of $X$ inherited from that of $\mathcal{L}_{\mathfrak{q}(\lambda_a),S}(Z_a)$ (c.f. \cite[Theorem 2.11]{SRV98}). 
Then the signatures of $\mathfrak{q}(\lambda_a)$-bottom layer $K$-types are preserved. More explicitly, one has
$$p^{Z_a}(\beta) = p^X(\mathcal{L}_{\mathfrak{q}(\lambda_a),S}^K(\beta)), \quad \quad q^{Z_a}(\beta) = q^X(\mathcal{L}_{\mathfrak{q}(\lambda_a),S}^K(\beta)).$$
\end{theorem}

\begin{remark} \label{rmk-bottomlayer}
    Indeed the above bottom layer $K$-type theorem works for any $\theta$-stable parabolic subalgebras $\mathfrak{q}' = \mathfrak{l}'+\mathfrak{u}'$. Namely, let $Z$ be a Hermitian $(\mathfrak{l}',L'\cap K)$-module such that $X$ is a lowest $K$-type subquotient of the cohomologically induced module $\mathcal{L}_{\mathfrak{q}',S}(Z)$, 
    and $\beta'$ be a $L' \cap K$-type in $Z$ such that $\mu(\beta')+2\rho(\mathfrak{u}'\cap \mathfrak{p})$ is $K$-dominant.
    Then the above theorem also holds upon replacing $Z_a$, $\mathfrak{q}(\lambda_a)$ and $\beta$ with $Z$, $\mathfrak{q}'$ and $\beta'$ respectively (c.f. \cite[Theorem 6.34]{KV95}). 
\end{remark}

We end this subsection by the following definition, which will be used in our refined conjecture of
Salamanca-Riba and Vogan:
\begin{definition} \label{def-upto}
Let $X$ be an irreducible, Hermitian $(\mathfrak{g},K)$-module with lowest $K$-types $\{\delta_1, \dots, \delta_r\}$. Suppose
$\{\chi_1, \dots, \chi_u\}$ be the set of $K$-types appearing in the tensor products $\delta_i \otimes \mathfrak{p}$. 
We say $X$ is {\bf non-unitary up to level $\mathfrak{p}$} if
$$\left(\sum_i p^X(\delta_i) + \sum_j p^X(\chi_j)\right) \cdot \left(\sum_i q^X(\delta_i) + \sum_j q^X(\chi_j)\right) \neq 0$$
In other words, the Hermitian form has indefinite signatures on the $K$-types 
$$\{\delta_1, \dots, \delta_r, \chi_1, \dots, \chi_u\}.$$ 

Suppose furthermore that $\mathfrak{g} = \mathfrak{k} + \mathfrak{p}^+ + \mathfrak{p^-}$ is Hermitian symmetric, 
and $\{\chi_1^{+}, \dots, \chi_s^{+}\}$, $\{\chi_1^{-}, \dots, \chi_t^{-}\}$ be the set of $K$-types
appearing in the tensor products $\delta_i \otimes \mathfrak{p}^{+}$ and $\delta_i \otimes \mathfrak{p}^{-}$ respectively. We say $X$ is {\bf non-unitary up to level $\mathfrak{p}^{\pm}$} if
$$\left(\sum_i p^X(\delta_i) + \sum_j p^X(\chi_j^{\pm})\right) \cdot \left(\sum_i q^X(\delta_i) + \sum_j q^X(\chi_j^{\pm})\right) \neq 0.$$
\end{definition}

\subsection{Salamanca-Riba and Vogan's conjecture} \label{sec-SV}
In the Langlands classification using $\theta$-stable data, one would hope that the map \eqref{eq-lambdaa}
in Theorem \ref{thm-lambdaa}(a) is also a bijection upon restricting to unitary representation. Unfortunately, there are simple examples that it does not (see the paragragh after Theorem 2.11 of \cite{SRV98}). To remedy the problem, Salamanca-Riba and Vogan considered `enlarging' the theta-stable Levi subgroup or, equivalently, projecting more $\mu$ to $0$. 

\begin{definition}
For all $\mu \in \Delta^+(\mathfrak{k},\mathfrak{t})$, define $\lambda_u(\mu) := P(\mu + 2\rho(\mathfrak{k}) - 2\rho(\mathfrak{g}))$. 
\end{definition}

As in Theorem \ref{thm-lambdaa}, one also has
\begin{theorem}[\cite{SRV98} Section 3]
Let 
$$\Pi_h^{\lambda_u}(G) := \Big\{\begin{matrix} \pi\ \text{adm. irred. Hermitian}\\ 
(\mathfrak{g},K)\text{-module}\end{matrix}\ \Big|\ 
\begin{matrix}\text{a lowest K-type}\ \delta\ \text{of}\ \pi\\ 
\text{satisfies}\ \lambda_u(\delta) = \lambda_u \end{matrix}\Big\}.$$ 
Then there is a bijection 
$$\Pi_h^{\lambda_u}(G(\lambda_u)) \longrightarrow \Pi_h^{\lambda_u}(G).$$
\end{theorem}

The reason behind the introduction of $\lambda_u$ is as follows:
\begin{conjecture} \label{conj-lambdaunitary}
The above bijection preserves unitarity. 
\end{conjecture}

Assuming the conjecture holds, then one can reduce the study of $\widehat{G}$ to the representations
$X$ whose lowest $K$-types $\delta$ satisfies $G(\lambda_u(\delta)) = G$. Such $K$-types are called {\bf unitarily small}. 

\medskip
Since the announcement of \cite{SRV98}, there is nearly no progress on how to prove the above conjecture. Nevertheless, it was shown that the following conjecture would imply Conjecture \ref{conj-lambdaunitary}.
\begin{conjecture}[\cite{SRV98} Conjecture 5.7] \label{conj-original}
Let $X$ be an irreducible, Hermitian $(\mathfrak{g},K)$-module with a unitarily small lowest $K$-type $\delta$ and real infinitesimal character $\Lambda = (\lambda_a(\delta), \nu)$ $\in$ $\mathfrak{h}^*$. If $X$ is unitary, then 
$\Lambda$ must lie in the convex hull:
\begin{equation} \label{eq-hull}
\lambda_u(\delta) + (\text{convex hull of } W(\mathfrak{g},\mathfrak{h}) \cdot \rho(\mathfrak{g})).
\end{equation} 
Otherwise, the Hermitian form of $X$ has opposite signatures on two unitarily small $K$-types $\delta_1$, $\delta_2$ in $X$.
\end{conjecture}
In Section \ref{sec-general}, we will give a proof of this conjecture for $G=U(p,q)$.

\subsection{Vogan's FPP conjecture}
In \cite{V23}, Vogan proposed another possible reduction step in the study of the unitary dual of real reductive groups. Before stating the conjecture, we recall some results in \cite{KV95} on cohomological induction:
\begin{definition} \label{def-good}
    Let $\fq=\fl+\fu$ be a $\theta$-stable parabolic subalgebra of $\fg$, and $L$ be the normalizer of $\fq$ in $G$. An $(\fl, L\cap K)$-module $Z$ with infinitesimal character $\lambda_L$ is in the {\bf good range} if
\begin{equation}\label{good}
{\rm Re} \langle \lambda_L + \rho(\fu), \alpha \rangle > 0, \quad \forall \alpha\in \Delta(\fu, \fh).
\end{equation}
\end{definition}
By \cite[Theorem 0.50]{KV95}, suppose $Z$ is a $(\fl, K \cap L)$-module
in the good range, then the cohomologically induced modules $\mathcal{L
}_{\mathfrak{q},j}(Z) = 0$ for all
$j \neq S = \dim(\fu \cap \fp)$, and $\mathcal{L}_{\mathfrak{q},S}(Z)$ is irreducible if and only if $Z$ is irreducible. 

As for unitarity, it follows from \cite[Proposition 10.9]{V86} that if $Z$ is Hermitian, then $\mathcal{L}_{\mathfrak{q},S}(Z)$ is unitary if and only if $Z$ is unitary. Therefore, one is interested in understanding the irreducible modules that are not cohomologically induced from any proper $\theta$-stable parabolic subalgebras in the good range. 

In view of the above discussions, Vogan's Fundamental Parallelepiped (FPP) conjecture is stated as follows:
\begin{conjecture}[\cite{V23}] \label{conj-fpp}
    Let $X$ be an irreducible $(\mathfrak{g},K)$-module with real infinitesimal character $\Lambda \in \mathfrak{h}^*$, which is not cohomologically induced in the good range from any $(\mathfrak{l},L \cap K)$-module $X_L$ for any proper $\theta$-stable parabolic subalgebra $\mathfrak{q} = \mathfrak{l} + \mathfrak{u}$. If $X$ is unitary, then one must have
    $\langle \Lambda, \alpha^{\vee} \rangle \leq 1$ for all simple roots $\alpha \in \Delta^+(\mathfrak{g},\mathfrak{h})$.
\end{conjecture}
We will give a proof of this conjecture for $G=U(p,q)$ in Section \ref{sec-fpp}.

\section{Methods for detecting non-unitarity} \label{sec-nonunit}
 In this section, we recall some useful tools for detecting non-unitarity of $X$. They will be used extensively in the proof of the Salamanca-Riba-Vogan's conjecture (Conjecture \ref{conj-original}) as well as the FPP conjecture (Conjecture \ref{conj-fpp}) for $G = U(p,q)$. One also expects them to play an essential role in the proof of both conjectures for other real reductive groups.

\subsection{Parthasarathy's Dirac inequality} 
Parthasarathy's Dirac inequality is very effective in 
detecting non-unitary of various $(\mathfrak{g},K)$-modules. For instance, it is used heavily in \cite{SR88}
(for $SL(n,\mathbb{R})$, $U(p,q)$, $Sp(2n,\mathbb{R})$) and \cite{SR99} (general case)
to prove that all irreducible unitary $(\mathfrak{g},K)$-modules with regular, integral infinitesimal 
character must be an $A_{\mathfrak{q}}(\lambda)$-module. It is also used to prove a slightly weaker version of Conjecture \ref{conj-original} in \cite[Proposition 7.18]{SRV98}.
\begin{theorem}[\cite{SR88},  Lemma 6.1] \label{thm-partha}
Let $X$ be an irreducible, Hermitian $(\mathfrak{g},K)$-module with infinitesimal character $\Lambda$. Suppose there is a $K$-type $\delta$ appearing in $X$ with highest weight $\mu$, and a choice of positive root system $\Delta^+(\mathfrak{g},\mathfrak{t}) \supset \Delta^+(\mathfrak{k},\mathfrak{t})$ such that
\begin{equation} \label{eq-dirac}
\| \{\mu - \rho(\Delta^+(\mathfrak{p}))\} + \rho(\Delta^+(\mathfrak{k})) \| < \| \Lambda \|,
\end{equation}
where $\{\sigma\}$ is defined as the $W(\mathfrak{k},\mathfrak{t})$-conjugate of $\sigma$ making $\{\sigma\}$
$K$-dominant, then the Hermitian form of $X$ is indefinite on level $\mathfrak{p}$, that is, the form has opposite signature on $\delta$ and another $K$-type in $\delta \otimes \mathfrak{p}$.

Suppose furthermore that $\mathfrak{g} = \mathfrak{k} + \mathfrak{p}^+ + \mathfrak{p}^-$ is Hermitian symmetric, 
and there is a $K$-type $\delta$ appearing in $X$ with highest weight $\mu$, such that
\begin{equation} \label{eq-dirach}
\| \{\mu - \rho(\Delta(\mathfrak{p}^{\pm}))\} + \rho(\Delta(\mathfrak{k})) \| < \| \Lambda \|,
\end{equation}
then the Hermitian form is indefinite on level $\mathfrak{p}^{\mp}$.
\end{theorem}

As an application of the above theorem, if the infinitesimal character $\Lambda$ of $X$ is `too big', then the lowest $K$-types of $X$ would satisfy
\eqref{eq-dirac} or \eqref{eq-dirach}, implying that any such $X$ is non-unitary up to level $\mathfrak{p}$ or $\mathfrak{p}^{\pm}$.

\subsection{Jantzen filtration}
We will mainly follow \cite{V84} in this section. The theorem below comes \cite[Theorem 3.8]{V84}, which is essential in the proof of the main theorem:
\begin{theorem} \label{thm-jantzen}
Let 
$$I(t) := Ind_{MAN}^G(\sigma \boxtimes (\nu_0 + t\nu) \boxtimes 1),\quad t \in (t_0 - \epsilon, t_0 + \epsilon)$$ 
be a continuous family of parabolically induced modules with a nonzero invariant Hermitian form. Suppose the following holds:
\begin{itemize}
\item $I(t)$ is irreducible for all $t$ satisfying $0 < |t - t_0| < \epsilon$ ; and 
\item $I(t_0)$ is reducible with an irreducible submodule $J(t_0)$.
\end{itemize}
Then there is a Jantzen filtration (see Definition \cite[Definition 3.7]{V84} for details) of $I(t_0)$:
$$I(t_0) = I(t_0)^N \supset I(t_0)^{N-1} \supset \dots \supset I(t_0)^0 = 0$$
satisfying:
\begin{enumerate}
\item[(a)] $I(t_0)^1 = J(t_0)$;
\item[(b)] There is an non-degenerate invariant Hermitian form $\langle, \rangle_n$ on $I(t_0)^{n+1}/I(t_0)^n$ with signature $(p_n(\delta), q_n(\delta))$ for each $K$-type $\delta$. For instance, if $V_{\delta}$ does not appear in the quotient, then $(p_n(\delta), q_n(\delta)) = (0,0)$.
\item[(c)] For all $0 < \epsilon_0 \leq \epsilon$, the non-degenerate invariant Hermitian form on $I(t_0-\epsilon_0)$ has signature 
$$\left(\sum_{n=0}^{N-1} p_n(\delta),\ \sum_{n=0}^{N-1} q_n(\delta)\right)$$
for all $K$-types $\delta$.
\item[(d)] For all $0 < \epsilon_0 \leq \epsilon$, the non-degenerate invariant Hermitian form on $I(t_0+\epsilon_0)$ has signature 
$$\left(\sum_{m} p_{2m}((\delta) + q_{2m+1}(\delta)),\ \sum_{m} (p_{2m+1}(\delta) + q_{2m}(\delta))\right)$$
for all $K$-types $\delta$.
\end{enumerate}
\end{theorem}
Here is an immediate consequence of the above theorem:
\begin{corollary} \label{cor-jantzen}
Let $I(t)$, $t > 0$ be a family of real induced modules with a nonzero invariant Hermitian form. Suppose 
\begin{itemize} 
\item $I(t)$ is irreducible except at $t =  t_1 < t_2 < \dots$;
\item For all $i \in \mathbb{N}$, $I(t_i)$ has a unique irreducible submodule $J(t_i)$; 
\item There exists a $K$-type $\delta$ such that the multiplicities $[\delta : J(t_i)]$ are equal
for all $i$. In other words, the composition factors of $I(t_i)$ other than $J(t_i)$ does {\bf not} contain $\delta$.
\end{itemize}
Then the signature $(p(\delta),q(\delta))$ of $J(t)$ is constant for all $t > 0$.
\end{corollary} 
\begin{proof}
By hypothesis, $(p_n(\delta), q_n(\delta)) = (0,0)$ for all $n > 0$. By Theorem \ref{thm-jantzen}(d),
the signature of $\delta \in \widehat{K}$ remains unchanged upon passing from $t_i - \epsilon$ to $t_i + \epsilon$ for all $i$.
\end{proof}

To end this section, we briefly mention the general strategy for detecting non-unitarity of an irreducible $(\mathfrak{g},K)$-module $X$ in this manuscript. Firstly, suppose there exists a  $(\mathfrak{l}',L'\cap K)$-module $Z$ 
such that
$X$ is a lowest $K$-type subquotient of $\mathcal{L}_{\mathfrak{q}',S}(Z)$ (e.g. $L' = L$ and $Z = Z_a$ in Theorem \ref{thm-lambdaa}), and $Z$
has indefinite Hermitian forms on two bottom layer $L' \cap K$-types
$\beta_1, \beta_2$ up to level $(\mathfrak{l}' \cap \mathfrak{p})^{\pm}$, 
then Remark \ref{rmk-bottomlayer} implies that the $K$-types $\mathcal{L}_{\mathfrak{q}',S}^K(\beta_1), 
\mathcal{L}_{\mathfrak{q}',S}^K(\beta_2)$ also have indefinite forms on $X$ up to level $\mathfrak{p}^+$.


\smallskip
In the cases when the bottom layer $K$-type arguments do not work, we apply the {\it deformation arguments} which are widely used in \cite{BJ90a}, \cite{BC05} and \cite{BDW22} to study the unitary dual for various real and p-adic reductive groups. More explicitly,  
consider a family of Hermitian representations $X(t)$ for $t \geq 0$ with $X(0) = X$
by increasing some `big' $\nu$-coordinates of $X$ (see Proposition \ref{prop-parallel} and Theorem \ref{thm-semipm}). In such a case, the multiplicities of the $K$-types up to level $\mathfrak{p}^{\pm}$ remain unchanged for all $X(t)$. Then Corollary \ref{cor-jantzen} implies that the signature of such $K$-types remain unchanged as $t$ goes to $\infty$, so that one can invoke Parthasarathy's Dirac inequality (Theorem \ref{thm-partha}) to conclude that these $K$-types have indefinite signatures for large $t$, and hence for all $t \geq 0$.

\section{Representations of $U(p,q)$} \label{sec-upq}
In this section, we give a description of the information needed for the construction of all
irreducible $(\mathfrak{g},K)$-modules given in Algorithm \ref{alg-langlands} for $G = U(p,q)$.

\subsection{$\lambda_a$-blocks}
As in Steps (i) -- (ii) of Algorithm \ref{alg-langlands}, one needs to describe $\lambda_a(\delta)$
for all $\delta \in \widehat{K}$. 
The following combinatorial description of $\lambda_a$ was first introduced in \cite{SR88}. A similar description also appears in \cite{B04}.
\begin{definition}
Let $G = U(p,q)$ with $p + q \equiv \epsilon\ (\text{mod}\ 2)$ ($\epsilon = 0$ or $1$). 
A {\bf $\lambda_a$-block of size $(r,s)$ with content $\gamma$} (or simply a {\bf $\gamma$-block}),
where $0 \leq r \leq p$, $0 \leq s \leq q$ are non-negative integers satisfying $|r-s| \leq 1$, and $\gamma \in \frac{1}{2}\mathbb{Z}$ is of one of the following forms:
\begin{itemize}
\item {\bf Rectangle} of size $(r,r)$:
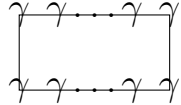
\begin{center}
\begin{tikzpicture}
\draw
    (0,0) node {\Large $\gamma$}
 -- (0.5,0) node {\Large $\gamma$}		
 -- (1,0) node {\Large $\ldots$}
 -- (1.5,0) node {\Large $\gamma$}
 -- (2,0) node {\Large $\gamma$}
 -- (2,-1) node {\Large $\gamma$}
 -- (1.5,-1) node {\Large $\gamma$}
-- (1,-1) node {\Large $\ldots$}
-- (0.5,-1) node {\Large $\gamma$}
-- (0,-1) node {\Large $\gamma$}
 -- cycle;
\end{tikzpicture} 
\end{center}
where $\gamma + \frac{\epsilon}{2} \in \mathbb{Z}$.

\item {\bf Parallelogram} of size $(r,r)$:
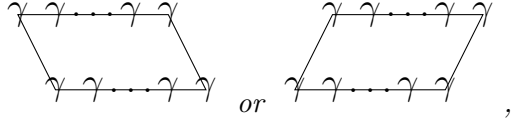
\begin{center}
\begin{tikzpicture}
\draw
    (0,0) node {\Large $\gamma$}
 -- (0.5,0) node {\Large $\gamma$}		
 -- (1,0) node {\Large $\ldots$}
 -- (1.5,0) node {\Large $\gamma$}
 -- (2,0) node {\Large $\gamma$}
 -- (2.5,-1) node {\Large $\gamma$}
 -- (2,-1) node {\Large $\gamma$}
-- (1.5,-1) node {\Large $\ldots$}
-- (1,-1) node {\Large $\gamma$}
-- (0.5,-1) node {\Large $\gamma$}
 -- cycle;
\end{tikzpicture} or 
\begin{tikzpicture}
\draw
    (0,0) node {\Large $\gamma$}
 -- (0.5,0) node {\Large $\gamma$}		
 -- (1,0) node {\Large $\ldots$}
 -- (1.5,0) node {\Large $\gamma$}
 -- (2,0) node {\Large $\gamma$}
 -- (1.5,-1) node {\Large $\gamma$}
 -- (1,-1) node {\Large $\gamma$}
-- (0.5,-1) node {\Large $\ldots$}
-- (0,-1) node {\Large $\gamma$}
-- (-0.5,-1) node {\Large $\gamma$}
 -- cycle;
\end{tikzpicture},
\end{center}
where $\gamma + \frac{\epsilon + 1}{2} \in \mathbb{Z}$.

\item {\bf Trapezoid} of size $(r,r-1)$ or $(r,r+1)$:
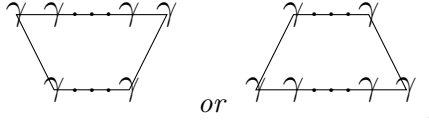
\begin{center}
\begin{tikzpicture}
\draw
    (0,0) node {\Large $\gamma$}
 -- (0.5,0) node {\Large $\gamma$}		
 -- (1,0) node {\Large $\ldots$}
 -- (1.5,0) node {\Large $\gamma$}
 -- (2,0) node {\Large $\gamma$}
 -- (1.5,-1) node {\Large $\gamma$}
-- (1,-1) node {\Large $\ldots$}
-- (0.5,-1) node {\Large $\gamma$}
 -- cycle;
\end{tikzpicture} or 
\begin{tikzpicture}
\draw
 (0.5,0) node {\Large $\gamma$}		
 -- (1,0) node {\Large $\ldots$}
 -- (1.5,0) node {\Large $\gamma$}
 -- (2,-1) node {\Large $\gamma$}
 -- (1.5,-1) node {\Large $\gamma$}
-- (1,-1) node {\Large $\ldots$}
-- (0.5,-1) node {\Large $\gamma$}
-- (0,-1) node {\Large $\gamma$}
 -- cycle;
\end{tikzpicture},
\end{center}
where $\gamma + \frac{\epsilon + 1}{2} \in \mathbb{Z}$.

\end{itemize}
\end{definition}

\subsection{$\lambda_a$-data and $K$-types}
We now relate each $\delta \in \widehat{K}$ to a union of $\lambda_a$-blocks called {\bf $\lambda_a$-datum} (Definition \ref{def-datum}). The main result in this section is Proposition \ref{prop-khat}, which gives a 1-1 correspondence between $\widehat{K}$ and all $\lambda_a$-data of $G$.
This is closely related to the well-known bijection between $\widehat{K}$ and all tempered representations with real infinitesimal characters.
\begin{definition} \label{def-datum}
A {\bf $\lambda_a$-datum for $G$} is a collection of $\gamma_i$-blocks of sizes $(r_i,s_i)$ such that $\sum r_i = p$, $\sum s_i = q$ and all $\gamma_i$ are distinct.
\end{definition}

It is immediate from \cite[Section 8]{SR88} that for each $\delta \in \widehat{K}$, $\lambda_a := \lambda_a(\delta)$ determines a unique 
$\lambda_a$-datum, which we are going to describe for the rest of this section: Let $\mu := \mu(\delta)$, then
$$\mu + 2\rho(\mathfrak{k}) = (x_1, x_2, \dots, x_p\ |\ y_1, y_2, \dots, y_q),\quad \quad x_i - x_{i+1}, y_j - y_{j+1} \geq 2\quad \forall\ i, j.$$
Rearrange the entries of $\rho = \left(\frac{p+q-1}{2}, \frac{p+q-3}{2}, \dots, -\frac{p+q-3}{2}, -\frac{p+q-1}{2}\right)$
into 
$$w\rho = (r_1, r_2, \dots, r_{p+q})$$ 
so that it is in the same order as $\mu + 2\rho(\mathfrak{k})$, and subtract it from
$\mu + 2\rho(\mathfrak{k})$. In other words,
the $i^{th}$-largest integer in $\mu + 2\rho(\mathfrak{k})$ will be subtracted by $\frac{p+q-2i+1}{2}$. 

There is an ambiguity in our algorithm above if $x_i = y_j$ for some $i$ and $j$. Suppose they are the $i^{th}$ and $(i+1)^{st}$
largest integers in $\mu + 2\rho(\mathfrak{k})$, then instead of subtracting $x_i$ and $y_j$ by $\frac{p+q-2i\pm 1}{2}$ and $\frac{p+q-2i\mp 1}{2}$,
we subtract both terms by $\frac{p+q-2i}{2}$ (this is due to the definition of $P$ in Equation \eqref{eq-lambdaa}).

One obtains $\lambda_a := \lambda_a(\delta)$ by applying the above algorithm to $\mu(\delta)$, as well as the the $\theta$-stable parabolic
subalgebra $\mathfrak{q}(\lambda_a)$ and the quasisimple Levi subgroup $L$ in Algorithm \ref{alg-langlands}.
\begin{example}[\cite{B04}, Example 2.3] \label{eg-barbasch}
Let $G = U(7,4)$ and $\mu = (2,2,2,2,2,2,2\ |\ 0, -3,-3,-4)$. Then
\begin{align*}
\mu + 2\rho(\mathfrak{k}) &= (2,2,2,2,2,2,2\ |\ 0, -3,-3,-4) + (6,4,2,0,-2,-4,-6\ |\ 3,1,-1,-3)\\ 
&= 
(8,6,4,2,0,-2,-4\ |\ 3,-2,-4,-7).\end{align*}
Note that the $6^{th}$ and $9^{th}$ coordinates, as well as the $7^{th}$ and $10^{th}$ coordinates
are equal. In this case, 
\begin{align*}
\lambda_a(\mu) &= (8,6,4,2,0,-2,-4\ |\ 3,-2,-4,-7) - \Bigg(5,4,3,1,0,{\bf \frac{-3}{2}},{\bf \frac{-7}{2}}\ \Bigg|\ 2,{\bf \frac{-3}{2}},{\bf \frac{-7}{2}},-5\Bigg)\\ 
&= \Bigg(3,2,1,1,0,\frac{-1}{2},\frac{-1}{2}\ \Bigg|\ 1,\frac{-1}{2},\frac{-1}{2},-2\Bigg)
\end{align*}
and hence its corresponding $\lambda_a$-datum is
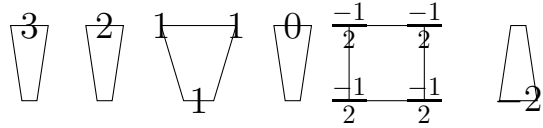
\begin{center}
\begin{tikzpicture}
\draw
    (0,0) 
 -- (0.25,0) node {\Large $3$}
-- (0.5,0) 
-- (0.35,-1) 
-- (0.15,-1) 
 -- cycle;

\draw
    (1,0) 
 -- (1.25,0) node {\Large $2$}
-- (1.5,0) 
-- (1.35,-1) 
-- (1.15,-1) 
 -- cycle;

\draw
    (2,0) node {\Large $1$}
 -- (3,0) node {\Large $1$}		
 -- (2.7,-1) 
-- (2.5,-1) node {\Large $1$}
-- (2.3,-1) 
 -- cycle;

\draw
    (3.5,0) 
 -- (3.75,0) node {\Large $0$}
-- (4,0) 
-- (3.85,-1) 
-- (3.65,-1) 
 -- cycle;

\draw
    (4.5,0) node {\Large $\frac{-1}{2}$}
 -- (5.5,0) node {\Large $\frac{-1}{2}$}
-- (5.5,-1) node {\Large $\frac{-1}{2}$}
-- (4.5,-1) node {\Large $\frac{-1}{2}$}
 -- cycle;

\draw
    (6.65,0) 
 -- (6.85,0) 
-- (7,-1) 
-- (6.75,-1) node {\Large $-2$}
-- (6.5,-1) 
 -- cycle;
\end{tikzpicture} 
\end{center}
\end{example}

There can be different $\delta$'s having the same $\lambda_a(\delta)$. Indeed, this happens precisely
when $\lambda_a(\mu)$ contains coordinates of the form
$$\lambda_a(\mu) = (\cdots, \underbrace{\gamma, \dots, \gamma}_{r\ entries}, \cdots\ |\ \cdots, \underbrace{\gamma, \dots, \gamma}_{r\ entries}, \cdots), \quad \gamma + \frac{\delta + 1}{2} \in \mathbb{Z},$$
Suppose the $r$ labeled entries on the left and right are the $(f_{\gamma}^++1)^{st} - (f_{\gamma}^++r)^{th}$ and 
$(f_{\gamma}^-+1)^{st} - (f_{\gamma}^-+r)^{th}$ coordinates respectively, then there
are exactly two choices of the corresponding coordinates of $\mu$:
\begin{equation} \label{eq-mupm}
\begin{aligned}
\mu_+ &= (\cdots, \underbrace{a, \dots, a}_{r\ entries}, \cdots\ |\ \cdots, \underbrace{b, \dots, b}_{r\ entries}, \cdots), \\
\mu_- &= (\cdots, \underbrace{a-1, \dots, a-1}_{r\ entries}, \cdots\ |\ \cdots, \underbrace{b+1, \dots, b+1}_{r\ entries}, \cdots)
\end{aligned}
\end{equation}
so that the $(f_{\gamma}^++1)^{st} - (f_{\gamma}^++r)^{th}$ and 
$(f_{\gamma}^-+1)^{st} - (f_{\gamma}^-+r)^{th}$ coordinates of $\lambda_a(\mu_{\pm})$ are equal to that of $\lambda_a(\mu)$. 
We assign $\mu_+$ with the $\lambda_a$-datum containing a $\gamma$-block of shape
\begin{tikzpicture}
\draw
    (0,0)  
 -- (0.5,0) 
 -- (0.75,-0.5) 
 -- (0.25,-0.5) 
 -- cycle;
\end{tikzpicture}, and $\mu_-$ with the $\lambda_a$-datum containing a $\gamma$-block of shape
\begin{tikzpicture}
\draw
 (0.25,0)	
 -- (0.75,0) 
 -- (0.5,-0.5) 
-- (0,-0.5) 
 -- cycle;
\end{tikzpicture}.
\begin{example} \label{eg-63}
Let $G = U(6,3)$ and $\mu_+ = ({\bf 0,0},-1,-1,-1,-1\ |\ {\bf 2,2},1)$ and $\mu_- = ({\bf -1,-1},-1,-1,-1,-1\ |\ {\bf 3,3},1)$.
One calculates easily that $\lambda_a(\mu_{\pm}) = ({\bf 1,1},0,0,-1,-2\ |\ {\bf 1,1}, 0)$ in both cases.
Under the correspondence described above, we have
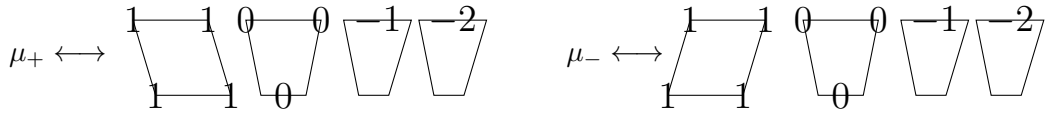
\begin{center}
\begin{tikzpicture}

\draw (-1,-0.5) node {$\mu_+ \longleftrightarrow$};

\draw
    (0,0) node {\Large $1$}
 -- (1,0) node {\Large $1$}
-- (1.3,-1) node {\Large $1$}
-- (0.3,-1) node {\Large $1$}
 -- cycle;

\draw
    (1.5,0) node {\Large $0$}
 -- (2.5,0) node {\Large $0$}
-- (2.3,-1) 
-- (2,-1) node {\Large $0$}
-- (1.7,-1) 
 -- cycle;

\draw
    (2.8,0) 
 -- (3.25,0) node {\Large $-1$}		
-- (3.7,0)
 -- (3.4,-1) 
-- (3,-1) 
 -- cycle;

\draw
    (3.8,0) 
 -- (4.25,0) node {\Large $-2$}		
-- (4.7,0)
 -- (4.4,-1) 
-- (4,-1) 
 -- cycle;
\end{tikzpicture} \quad \quad
\begin{tikzpicture}

\draw (-1,-0.5) node {$\mu_- \longleftrightarrow$};

\draw
    (0,0) node {\Large $1$}
 -- (1,0) node {\Large $1$}
-- (0.7,-1) node {\Large $1$}
-- (-0.3,-1) node {\Large $1$}
 -- cycle;

\draw
    (1.5,0) node {\Large $0$}
 -- (2.5,0) node {\Large $0$}
-- (2.3,-1) 
-- (2,-1) node {\Large $0$}
-- (1.7,-1) 
 -- cycle;

\draw
    (2.8,0) 
 -- (3.25,0) node {\Large $-1$}		
-- (3.7,0)
 -- (3.4,-1) 
-- (3,-1) 
 -- cycle;

\draw
    (3.8,0) 
 -- (4.25,0) node {\Large $-2$}		
-- (4.7,0)
 -- (4.4,-1) 
-- (4,-1) 
 -- cycle;
\end{tikzpicture}.
\end{center}
\end{example}

The above discussions leads to:
\begin{proposition}[\cite{SR88}, Proposition 8.2] \label{prop-khat}
Let $G = U(p,q)$. Then there is a one-to-one correspondence between $\widehat{K}$ and
the set of all $\lambda_a$-data of $G$.
\end{proposition}
\begin{proof}
The map from $\delta \in \widehat{K}$ to a $\lambda_a$-datum was described explicitly above.
As for the inverse, one can apply the following algorithm, which
is a restatement of Proposition 2.6 of \cite{SRV98}:
\begin{itemize}
\item[(i)] For each $\lambda_a$-datum, subtract the content of each $\gamma$-block by $\rho(\mathfrak{u}(\lambda_a))$.
\item[(ii)] For each parallelogram block of shape \begin{tikzpicture} \draw
    (0,0)  
 -- (0.5,0) 
 -- (0.75,-0.5) 
 -- (0.25,-0.5) 
 -- cycle;
\end{tikzpicture}, add all top entries of the block by $1/2$ and subtract all 
bottom entries of the block by $1/2$.
\item[(iii)] For each parallelogram block of shape \begin{tikzpicture}
\draw
 (0.25,0)	
 -- (0.75,0) 
 -- (0.5,-0.5) 
-- (0,-0.5) 
 -- cycle;
\end{tikzpicture}, subtract all top entries of the block by $1/2$ and add all 
bottom entries of the block by $1/2$.
\item[(iv)] Add $2\rho(\mathfrak{u} \cap \mathfrak{p})$ to the result in (iii)
(this corresponds to the functor $\mathcal{L}_{\mathfrak{q},S}^K(\bullet)$ in Theorem \ref{thm-lambdaa}(b)).
\end{itemize} 
\end{proof}
\begin{example}
Recall the $\lambda_a$-datum
\begin{center}
\begin{tikzpicture}
\draw
    (0,0) node {\Large $1$}
 -- (1,0) node {\Large $1$}
-- (0.7,-1) node {\Large $1$}
-- (-0.3,-1) node {\Large $1$}
 -- cycle;

\draw
    (1.5,0) node {\Large $0$}
 -- (2.5,0) node {\Large $0$}
-- (2.3,-1) 
-- (2,-1) node {\Large $0$}
-- (1.7,-1) 
 -- cycle;

\draw
    (2.8,0) 
 -- (3.25,0) node {\Large $-1$}		
-- (3.7,0)
 -- (3.4,-1) 
-- (3,-1) 
 -- cycle;

\draw
    (3.8,0) 
 -- (4.25,0) node {\Large $-2$}		
-- (4.7,0)
 -- (4.4,-1) 
-- (4,-1) 
 -- cycle;
\end{tikzpicture}
\end{center}
in Example \ref{eg-63}. Then $\rho(\mathfrak{u}) = \frac{1}{2}(5,5,-2,-2,-6,-8|5,5,-2)$ and hence
$$\lambda_a' = (1,1,0,0,-1,-2|1,1,0) - \rho(\mathfrak{u}) = \Big(\frac{-3}{2},\frac{-3}{2},1,1,2,2 \Big|\frac{-3}{2},\frac{-3}{2},1\Big)$$
By the shape of the parallelogram, and $2\rho(\mathfrak{u} \cap \mathfrak{p}) = (1,1,-2,-2,-3,-3|4,4,0)$, one has
$$\mu = \Big(\frac{-3}{2}-\frac{1}{2},\frac{-3}{2}-\frac{1}{2},1,1,2,2 \Big|\frac{-3}{2}+\frac{1}{2},\frac{-3}{2}+\frac{1}{2},1\Big) + (1,1,-2,-2,-3,-3|4,4,0)$$
which is precisely the $\mu_- = (-1,-1,-1,-1,-1,-1|3,3,1)$ in Example \ref{eg-63}.
\end{example}

\subsection{Langlands classification}
As discussed in the previous section, each $\lambda_a$-datum determines a $\delta \in \widehat{K}$, $\lambda_a = \lambda_a(\delta)$ and $L = G(\lambda_a)$ in Algorithm \ref{alg-langlands}. The only missing piece is the $\nu \in \mathfrak{a}_0^*$. Therefore we have make the following definition:
\begin{definition} \label{def-comb}
Let $G = U(p,q)$, and $X$
be an irreducible, Hermitian $(\mathfrak{g},K)$-module representation with real infinitesimal character
$\Lambda = (\lambda_a(\mu), \nu)$ and a lowest $K$-type $\delta$. A {\bf combinatorial $\theta$-stable datum}
attached to $X$ is given the following two components:
\begin{itemize}
\item[(i)] A $\lambda_a$-datum determined by $\delta$ under the bijection in Proposition \ref{prop-khat}; and
\item[(ii)] For each $\gamma_i$-block of size $(r_i,s_i)$ in (1), an element 
$\nu_i \in \mathbb{R}^{\min\{r_i,s_i\}}$ of the form $\nu_i = (\nu_{i,1} \geq \nu_{i,2} \geq \dots \geq -\nu_{i,2} \geq -\nu_{i,1})$ (we say $\nu_i$ be the {\bf $\nu$-coordinates corresponding to the $\gamma_i$-block}), so that $\nu = (\nu_1; \nu_2; \dots)$ up to conjugacy.
\end{itemize}
\end{definition}

Each combinatorial datum attached to $X$ determines an induced module \eqref{eq-langlands},
so that $X$ is the irreducible subquotient of the induced module containing $\delta$. 
As discussed after Algorithm \ref{alg-langlands}, the value of $\nu$ in (ii) determines
whether $X$ contains another lowest $K$-type $\delta' \neq \delta$. If both $\delta'$ and $\delta$
appear in $X$, one can apply $\delta'$ instead of $\delta$ in the above definition, and
obtain another combinatorial data attached to the same $X$. 
We will describe explicitly when this will happen in Example \ref{eg-lkts}.

\begin{example}
Let $G = U(6,2)$. The combinatorial $\theta$-stable datum attached to the trivial representation is of the form:
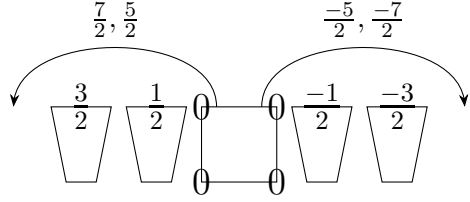
\begin{center}
\begin{tikzpicture}
\draw
    (0,0) 
 -- (0.4,0) node {\Large $\frac{3}{2}$}
-- (0.8,0) 
-- (0.6,-1) 
-- (0.2,-1) 
 -- cycle;

\draw
    (1,0) 
 -- (1.4,0) node {\Large $\frac{1}{2}$}
-- (1.8,0) 
-- (1.6,-1) 
-- (1.2,-1) 
 -- cycle;

\draw
    (2,0) node {\Large $0$}
 -- (3,0) node {\Large $0$}		
 -- (3,-1) node {\Large $0$}
-- (2,-1) node {\Large $0$}
 -- cycle;

\draw[arrows = {-Stealth[]}]          (2.2,0)   to [out=90,in=90]node[above]{$\frac{7}{2}, \frac{5}{2}$} (-0.5,0);
\draw[arrows = {-Stealth[]}]          (2.8,0)   to [out=90,in=90]node[above]{$\frac{-5}{2}, \frac{-7}{2}$} (5.5,0);

\draw
    (3.2,0) 
 -- (3.6,0) node {\Large $\frac{-1}{2}$}
-- (4,0) 
-- (3.8,-1) 
-- (3.4,-1) 
 -- cycle;

\draw
    (4.2,0) 
 -- (4.6,0) node {\Large $\frac{-3}{2}$}
-- (5,0) 
-- (4.8,-1) 
-- (4.4,-1) 
 -- cycle;

\end{tikzpicture} 
\end{center}
More precisely, one computes that $\lambda_a(\mathrm{triv}) = (\frac{3}{2}, \frac{1}{2}, 0,0,0,0, \frac{-1}{2},\frac{-3}{2})$. Since the infinitesimal character of $\mathrm{triv}$ is $\Lambda = \rho$,
$\nu$ must be equal to $(0,0,\frac{7}{2},\frac{5}{2},\frac{-5}{2},\frac{-7}{2},0,0)$.

\medskip
As another example, the trivial representation of  $U(5,2)$ has combinatorial $\theta$-stable datum:
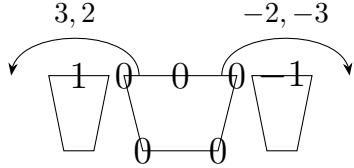
\begin{center}
\begin{tikzpicture}
\draw
    (0,0) 
 -- (0.4,0) node {\Large $1$}
-- (0.8,0) 
-- (0.6,-1) 
-- (0.2,-1) 
 -- cycle;

\draw
    (1,0)  node {\Large $0$}
 -- (1.75,0) node {\Large $0$}		
 -- (2.5,0) node {\Large $0$}		
 -- (2.25,-1) node {\Large $0$}
-- (1.25,-1) node {\Large $0$}
 -- cycle;

\draw[arrows = {-Stealth[]}]          (1.2,0)   to [out=90,in=90]node[above]{$3,2$} (-0.5,0);
\draw[arrows = {-Stealth[]}]          (2.3,0)   to [out=90,in=90]node[above]{$-2,-3$} (4,0);

\draw
    (2.7,0) 
 -- (3.1,0) node {\Large $-1$}
-- (3.5,0) 
-- (3.3,-1) 
-- (2.9,-1) 
 -- cycle;
\end{tikzpicture} 
\end{center}
\end{example}

\begin{example} \label{eg-lkts}
We begin by constructing irreducible $(\mathfrak{g},K)$-modules in $U(1,1)$ 
having more than one lowest $K$-type. Firstly, all such $K$-types must 
have the same $\lambda_a$-value by Theorem \ref{thm-lambdaa}(a). By Proposition \ref{prop-khat},
this occurs only when the $\lambda_a$-block is a parallelogram with $\gamma \in \mathbb{Z} + \frac{1}{2}$.

More precisely, the $K$-types $\delta_1 = V_{(k+1|k)}$ and $\delta_2 = V_{(k|k+1)}$ have $\lambda_a = (\frac{k+1}{2},\frac{k+1}{2})$ and $L = G(\lambda_a) = G$. Then $\gamma := \delta_1|_{T'} = \delta_2|_{T'} = \det^{2k+1}$ for $i = 1,2$, and
hence the two combinatorial $\theta$-stable data
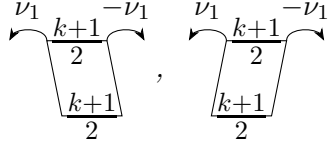
\begin{center}
\begin{tikzpicture}
\draw

(1.5,-0.5) node {,}

    (0,0)  node (1) {}
 -- (0.4,0) node {\Large $\frac{k+1}{2}$}
-- (0.8,0) node (2) {}
-- (1,-1) 
-- (0.6,-1) node {\Large $\frac{k+1}{2}$}
-- (0.2,-1)
 -- cycle;
\draw[arrows = {-Stealth[]}]          (0,0)   to [out=90,in=90]node[above]{$\nu_1$} (-0.5,0);
\draw[arrows = {-Stealth[]}]          (0.8,0)   to [out=90,in=90]node[above]{$-\nu_1$} (1.3,0);
\end{tikzpicture} \begin{tikzpicture}
\draw
    (0,0) node (3) {}
 -- (0.4,0) node {\Large $\frac{k+1}{2}$}
-- (0.8,0) node (4) {}
-- (0.6,-1) 
-- (0.2,-1) node {\Large $\frac{k+1}{2}$}
-- (-0.2,-1)
 -- cycle;
\draw[arrows = {-Stealth[]}]          (0,0)   to [out=90,in=90]node[above]{$\nu_1$} (-0.5,0);
\draw[arrows = {-Stealth[]}]          (0.8,0)   to [out=90,in=90]node[above]{$-\nu_1$} (1.3,0);
\end{tikzpicture}
\end{center}
both correspond to the same principal series representation $\pi_{\nu} = Ind_{T'A'N'}^G(\gamma \boxtimes \nu \boxtimes 1).$

\smallskip
If $\nu = (\nu_1, -\nu_1)$ with $\nu_1 \neq 0$, then $\pi_{\nu}$ is irreducible, and the {\bf two} combinatorial data
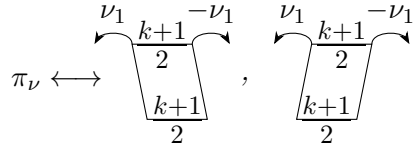
\begin{center}
\begin{tikzpicture}
\draw
(-1,-0.5) node {$\pi_{\nu} \longleftrightarrow$}
(1.5,-0.5) node {,}

    (0,0)  node (1) {}
 -- (0.4,0) node {\Large $\frac{k+1}{2}$}
-- (0.8,0) node (2) {}
-- (1,-1) 
-- (0.6,-1) node {\Large $\frac{k+1}{2}$}
-- (0.2,-1)
 -- cycle;
\draw[arrows = {-Stealth[]}]          (0,0)   to [out=90,in=90]node[above]{$\nu_1$} (-0.5,0);
\draw[arrows = {-Stealth[]}]          (0.8,0)   to [out=90,in=90]node[above]{$-\nu_1$} (1.3,0);
\end{tikzpicture} \begin{tikzpicture}
\draw
    (0,0) node (3) {}
 -- (0.4,0) node {\Large $\frac{k+1}{2}$}
-- (0.8,0) node (4) {}
-- (0.6,-1) 
-- (0.2,-1) node {\Large $\frac{k+1}{2}$}
-- (-0.2,-1)
 -- cycle;
\draw[arrows = {-Stealth[]}]          (0,0)   to [out=90,in=90]node[above]{$\nu_1$} (-0.5,0);
\draw[arrows = {-Stealth[]}]          (0.8,0)   to [out=90,in=90]node[above]{$-\nu_1$} (1.3,0);
\end{tikzpicture}
\end{center}
correspond to the {\bf same} irreducible module $\pi_{\nu}$ (with two lowest $K$-types).

\smallskip
If $\nu = (0,0)$, it is well-known that $\pi_{0} = \pi_{ds}^{+} \oplus \pi_{ds}^{-}$ splits into $2$ `limits of discrete series' with $K$-spectra $\pi_{ds}^+|_{K} := \bigoplus_{m \in \mathbb{N}} V_{(k+m+1|k-m)}$
and $\pi_{ds}^-|_{K} := \bigoplus_{m \in \mathbb{N}} V_{(k-m|k+m+1)}$. So we have
\begin{center}
\begin{tikzpicture}
\draw
(-1,-0.5) node {$\pi_{ds}^+ \longleftrightarrow$}

    (0,0)  node (1) {}
 -- (0.4,0) node {\Large $\frac{k+1}{2}$}
-- (0.8,0) node (2) {}
-- (1,-1) 
-- (0.6,-1) node {\Large $\frac{k+1}{2}$}
-- (0.2,-1)
 -- cycle;
\path (1) edge     [loop above]       node[] {$0$}         (1);
\path (2) edge     [loop above]       node[] {$0$}         (2);
\end{tikzpicture} \quad \quad  \quad
\begin{tikzpicture}
\draw

(-1,-0.5) node {$\pi_{ds}^- \longleftrightarrow$}

    (0,0) node (3) {}
 -- (0.4,0) node {\Large $\frac{k+1}{2}$}
-- (0.8,0) node (4) {}
-- (0.6,-1) 
-- (0.2,-1) node {\Large $\frac{k+1}{2}$}
-- (-0.2,-1)
 -- cycle;
\path (3) edge     [loop above]       node[] {$0$}         (3);
\path (4) edge     [loop above]       node[] {$0$}         (4);
\end{tikzpicture}
\end{center}
and the matching is unique, since the limits of discrete series has unique lowest $K$-types.

\smallskip
In general, an irreducible module $X$ corresponds to more than one combinatorial data if and only if 
one of its combinatorial data contains a parallelogram $\gamma$-block
with all $\nu$-coordinates not equal to zero (this follows immediately from the the calculations of $R$-groups in \cite{V79}). In such a case, all other combinatorial data corresponding to $X$ can be obtained by reflecting between parallelograms 
with all nonzero $\nu$-coordinates:
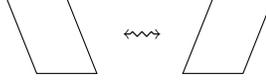
\begin{center}
\begin{tikzpicture}
 \draw   
(1.8,-0.5) node {$\leftrightsquigarrow$}

(0,0)  
-- (0.8,0)
-- (1.2,-1) 
-- (0.4,-1)
 -- cycle;

\end{tikzpicture} \begin{tikzpicture}
\draw    (0,0)  
-- (0.8,0)
-- (0.4,-1) 
-- (-0.4,-1)
 -- cycle;
\end{tikzpicture}
\end{center}

As a consequence, one can also determine {\bf all} lowest $K$-types of $X$, the total number of which is equal to $2^{\#(\text{parallelogram }\gamma \text{-blocks with all nonzero } \nu \text{ entries})}$. 

In Remark \ref{rmk-parallel} below, we will also determine the signatures of the Hermitian form on these lowest $K$-types.
\end{example}


\begin{definition} \label{def-bottom}
Let $X$ be an irreducible admissible $(\mathfrak{g},K)$-module with a lowest $K$-type $\delta$. Suppose
there is a $\gamma$-block of size $(r,s)$ in its corresponding combinatorial $\theta$-stable datum, i.e. 
$\lambda_a(\delta)$ is of the form
$$\lambda_a(\delta) = (\cdots, \underbrace{\gamma, \dots, \gamma}_{\begin{matrix} (f^+_{\gamma}+1)^{st}\ to\ (f^+_{\gamma}+r)^{th}\\ coordinates \end{matrix}}, \cdots| 
\cdots, \underbrace{\gamma, \dots, \gamma}_{\begin{matrix} (f^-_{\gamma}+1)^{st}\ to\ (f^-_{\gamma}+s)^{th}\\ coordinates \end{matrix}}, \cdots).$$
\begin{itemize}
\item[(a)] The {\bf $\gamma$-component of a lowest $K$-type $\delta$} are the $(f^+_{\gamma}+1)^{st}$ -- $(f^+_{\gamma}+r)^{th}$ and
$(f^-_{\gamma}+1)^{st}$ -- $(f^+_{\gamma}+s)^{th}$ coordinates of its corresponding highest weight $\mu = \mu(\delta)$:
\begin{align*}
\mu = (\mu_1^+, \cdots, \underline{\mu_{f^+_{\gamma}+1}^+, \dots, \mu_{f^+_{\gamma}+r}^+}, \cdots, \mu_p^+| \mu_1^-, \cdots, \underline{\mu_{f^-_{\gamma}+1}^-, \dots, \mu_{f^-_{\gamma}+s}^-}, \cdots, \mu_q^-).
\end{align*}
\item[(b)] We say $X$ is {\bf $\mathfrak{p}^+$-bottom layer on the $\gamma$-block} if 
$$\mu_{f^+_{\gamma}}^+ > \mu_{f^+_{\gamma}+1}^+ \quad \quad \text{and} \quad \quad \mu_{f^-_{\gamma}+s}^- > \mu_{f^-_{\gamma}+s+1}^-$$
\item[(c)] Similarly, we say $X$ is {\bf $\mathfrak{p}^-$-bottom layer on the $\gamma$-block} if 
$$\mu_{f^+_{\gamma}+r}^+ > \mu_{f^+_{\gamma}+r+1}^+ \quad \quad \text{and} \quad \quad \mu_{f^-_{\gamma}}^- > \mu_{f^-_{\gamma}+1}^-.$$
\end{itemize}
\end{definition}

As a consequence of Theorem \ref{thm-bottomlayer}, we have:
\begin{corollary} \label{cor-bottom}
Let $X$ be an irreducible representation corresponding to a combinatorial $\theta$-stable datum. Suppose
$X$ is $\mathfrak{p}^{\pm}$-bottom layer on a $\gamma$-block in its combinatorial datum, and the $U(r,s)$-representation corresponding to this single $\gamma$-block is non-unitary up to level $\mathfrak{p}^{\pm}$ (Definition \ref{def-upto}), 
then $X$ is also non-unitary up to level $\mathfrak{p}^{\pm}$.
\end{corollary}

In other words, under the bottom layer hypothesis, one can detect non-unitarity of $X$ by just looking at a single $\gamma$-block in (one of) its combinatorial datum. This will be a very effective tool in proving the main theorem.

We now give a necessary and sufficient condition to determine whether $X$ is $\mathfrak{p}^{\pm}$-bottom layer on a certain $\gamma$-block by simply looking at the shapes of the blocks in the combinatorial datum of $X$.
\begin{proposition} \label{prop-bottom}
Let $X$ be a $(\mathfrak{g},K)$-module whose combinatorial $\theta$-stable datum contains a $\gamma$-block. Then
\begin{itemize}
\item[(a)] $X$ is $\mathfrak{p}^+$-bottom layer on the $\gamma$-block iff the {\bf top-left} corner of the $\gamma$-block is 
{\bf NOT} of the form:
\begin{center}
\begin{tikzpicture}
\draw
    (1.5,0) 
-- (2,0)
 -- (1.5,-1)
 -- (1,-1);

\draw
    (2.7,0) 
-- (2.2,0) node {$\bullet$}
 -- (2.2,-1)
 -- (2.7,-1);  

\draw (0.5,-0.5) node {$(\gamma+ \frac{1}{2})$-block};
\draw (2.8,-0.5) node {$\gamma$-block};
\end{tikzpicture},
\begin{tikzpicture}
\draw
    (1.5,0) 
-- (2,0)
 -- (1.5,-1)
 -- (1,-1);

\draw
    (2.7,0) 
-- (2.2,0) node {$\bullet$}
 -- (2.7,-1)
 -- (3.2,-1);  

\draw (0.5,-0.5) node {$(\gamma+1)$-block};
\draw (3.2,-0.5) node {$\gamma$-block};

\end{tikzpicture},
\begin{tikzpicture}
\draw
    (1.5,0) 
-- (2,0)
 -- (2,-1)
-- (1.5,-1);

\draw
    (2.7,0) 
-- (2.2,0) node {$\bullet$}
 -- (2.7,-1)
 -- (3.2,-1);  

\draw (0.7,-0.5) node {$(\gamma+ \frac{1}{2})$-block};
\draw (3.2,-0.5) node {$\gamma$-block};
\end{tikzpicture}
\end{center}

and the {\bf bottom-right} corner of the $\gamma$-block is {\bf NOT} of the form:
\begin{center}
\begin{tikzpicture}
\draw
    (1,0) 
-- (1.5,0)
 -- (1.5,-1) node {$\bullet$}
 -- (1,-1);

\draw
    (2.7,0) 
-- (2.2,0)
 -- (1.7,-1)
 -- (2.2,-1);  

\draw (3.2,-0.5) node {$(\gamma - \frac{1}{2})$-block};
\draw (0.85,-0.5) node {$\gamma$-block};
\end{tikzpicture}, 
\begin{tikzpicture}
\draw
    (1,0) 
-- (1.5,0)
 -- (2,-1) node {$\bullet$}
 -- (1.5,-1);

\draw
    (3.2,0) 
-- (2.7,0)
 -- (2.2,-1)
 -- (2.7,-1);  

\draw (3.7,-0.5) node {$(\gamma - 1)$-block};
\draw (1,-0.5) node {$\gamma$-block};

\end{tikzpicture}, 
\begin{tikzpicture}
\draw
    (1.5,0) 
-- (2,0)
 -- (2.5,-1) node {$\bullet$}
-- (2,-1);

\draw
    (3.2,0) 
-- (2.7,0)
 -- (2.7,-1)
 -- (3.2,-1);  

\draw (3.9,-0.5) node {$(\gamma - \frac{1}{2})$-block};
\draw (1.5,-0.5) node {$\gamma$-block};
\end{tikzpicture}
\end{center}

\item[(b)] $X$ is $\mathfrak{p}^-$-bottom layer on the $\gamma$-block iff the {\bf top-right} corner of the $\gamma$-block is 
{\bf NOT} of the form:
\begin{center}
 \begin{tikzpicture}
\draw
    (1.5,0) 
-- (2,0) node {$\bullet$}
 -- (2,-1)
-- (1.5,-1);

\draw
    (2.7,0) 
-- (2.2,0)
 -- (2.7,-1)
 -- (3.2,-1);  

\draw (3.55,-0.5) node {$(\gamma - \frac{1}{2})$-block};
\draw (1.3,-0.5) node {$\gamma$-block};
\end{tikzpicture},
\begin{tikzpicture}
\draw
    (1.5,0) 
-- (2,0) node {$\bullet$}
 -- (1.5,-1)
 -- (1,-1);

\draw
    (2.7,0) 
-- (2.2,0)
 -- (2.7,-1)
 -- (3.2,-1);  

\draw (3.55,-0.5) node {$(\gamma - 1)$-block};
\draw (1,-0.5) node {$\gamma$-block};
\end{tikzpicture}, 
\begin{tikzpicture}
\draw
    (1.5,0) 
-- (2,0) node {$\bullet$}
 -- (1.5,-1)
 -- (1,-1);

\draw
    (2.7,0) 
-- (2.2,0)
 -- (2.2,-1)
 -- (2.7,-1);  
\draw (3.3,-0.5) node {$(\gamma - \frac{1}{2})$-block};
\draw (1,-0.5) node {$\gamma$-block};
\end{tikzpicture},
\end{center}

and the {\bf bottom-left} corner of the $\gamma$-block is {\bf NOT} of the form:
\begin{center}
\begin{tikzpicture}
\draw
    (1.5,0) 
-- (2,0)
 -- (2.5,-1)
-- (2,-1);

\draw
    (3.2,0) 
-- (2.7,0)
 -- (2.7,-1) node {$\bullet$}
 -- (3.2,-1);  

\draw (1.1,-0.5) node {$(\gamma + \frac{1}{2})$-block};
\draw (3.3,-0.5) node {$\gamma$-block};
\end{tikzpicture}, 
\begin{tikzpicture}
\draw
    (1,0) 
-- (1.5,0)
 -- (2,-1)
 -- (1.5,-1);

\draw
    (3.2,0) 
-- (2.7,0)
 -- (2.2,-1) node {$\bullet$}
 -- (2.7,-1);  

\draw (0.6,-0.5) node {$(\gamma + 1)$-block};
\draw (3.1,-0.5) node {$\gamma$-block};
\end{tikzpicture}, \quad 
\begin{tikzpicture}
\draw
    (1,0) 
-- (1.5,0)
 -- (1.5,-1)
 -- (1,-1);

\draw
    (2.7,0) 
-- (2.2,0)
 -- (1.7,-1) node {$\bullet$}
 -- (2.2,-1);  

\draw (0.3,-0.5) node {$(\gamma + \frac{1}{2})$-block};
\draw (2.6,-0.5) node {$\gamma$-block};
\end{tikzpicture}
\end{center}
\end{itemize}
\end{proposition}
\begin{proof}
This follows easily from the calculations of $\delta \in \widehat{K}$ from its $\lambda_a$-data given in Proposition \ref{prop-khat}.
\end{proof}

\begin{remark} \label{rmk-blocks}
    By Remark \ref{rmk-bottomlayer}, the results in Corollary \ref{cor-bottom} and Proposition \ref{prop-bottom} can be generalized to a union of neighboring $\gamma$-blocks instead of a single $\gamma$-block. More explicitly, one can analogously apply Proposition \ref{prop-bottom} to the leftmost and rightmost block in the union to detect non-unitarity of $X$ up to level $\mathfrak{p}^{\pm}$ (see Example \ref{eg-barbasch2} below for more details).
\end{remark}

\begin{example} \label{eg-barbasch2}
Let $G = U(7,4)$. We continue with Example \ref{eg-barbasch}. Let $X$ be an irreducible representation
with lowest $K$-type $\delta \in \widehat{K}$ such that
$$\mu(\delta) = (2,2,2,2,2,2,2\ |\ 0, -3,-3,-4),\quad \quad \lambda_a(\delta) = \Bigg(3,2,1,1,0,\frac{-1}{2},\frac{-1}{2}\ \Bigg|\ 1,\frac{-1}{2},\frac{-1}{2},-2\Bigg)$$
and $\lambda_a$-datum
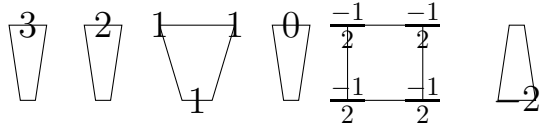
\begin{center}
\begin{tikzpicture}
\draw
    (0,0) 
 -- (0.25,0) node {\Large $3$}
-- (0.5,0) 
-- (0.35,-1) 
-- (0.15,-1) 
 -- cycle;

\draw
    (1,0) 
 -- (1.25,0) node {\Large $2$}
-- (1.5,0) 
-- (1.35,-1) 
-- (1.15,-1) 
 -- cycle;

\draw
    (2,0) node {\Large $1$}
 -- (3,0) node {\Large $1$}		
 -- (2.7,-1) 
-- (2.5,-1) node {\Large $1$}
-- (2.3,-1) 
 -- cycle;

\draw
    (3.5,0) 
 -- (3.75,0) node {\Large $0$}
-- (4,0) 
-- (3.85,-1) 
-- (3.65,-1) 
 -- cycle;

\draw
    (4.5,0) node {\Large $\frac{-1}{2}$}
 -- (5.5,0) node {\Large $\frac{-1}{2}$}
-- (5.5,-1) node {\Large $\frac{-1}{2}$}
-- (4.5,-1) node {\Large $\frac{-1}{2}$}
 -- cycle;

\draw
    (6.65,0) 
 -- (6.85,0) 
-- (7,-1) 
-- (6.75,-1) node {\Large $-2$}
-- (6.5,-1) 
 -- cycle;
\end{tikzpicture} 
\end{center}
Then the union of $3$-block, $2$-block and $1$-block correspond to the highlighted coordinates of $\mu(\delta)$:
$$\mu(\delta) = ({\bf 2,2,2,2}, 2,2,2\ |\ {\bf 0}, -3,-3,-4).$$
Since $({\bf 2+1,2,2,2},2,2,2\ |\ {\bf 0-1},-3,-3,-4) = ({\bf 3,2,2,2},2,2,2\ |\ {\bf -1}, -3,-3,-4)$
is a valid dominant $K$-weight,
$X$ is $\mathfrak{p}^+$-bottom layer on the union of $3$-to-$1$-blocks. 

This can also be seen by looking at
the leftmost block ($3$-block) and the rightmost block ($1$-block):
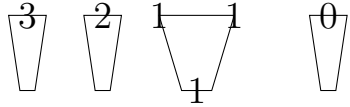
\begin{center}
\begin{tikzpicture}
\draw
    (0,0) 
 -- (0.25,0) node {\Large $3$}
-- (0.5,0) 
-- (0.35,-1) 
-- (0.15,-1) 
 -- cycle;

\draw
    (1,0) 
 -- (1.25,0) node {\Large $2$}
-- (1.5,0) 
-- (1.35,-1) 
-- (1.15,-1) 
 -- cycle;

\draw
    (2,0) node {\Large $1$}
 -- (3,0) node {\Large $1$}		
 -- (2.7,-1) 
-- (2.5,-1) node {\Large $1$}
-- (2.3,-1) 
 -- cycle;

\draw
    (4,0) 
 -- (4.25,0) node {\Large $0$}
-- (4.5,0) 
-- (4.35,-1) 
-- (4.15,-1) 
 -- cycle;
\end{tikzpicture} 
\end{center} 
Here the $3$-block satisfies the top-left hypothesis of 
Proposition \ref{prop-bottom}(a) (since there are no $\gamma$-blocks on the left of the $3$-block), and the 
$1$-block satisfies the bottom-right hypothesis of Proposition \ref{prop-bottom}(a). 

As a consequence, if the combinatorial sub-datum of $X$ with $\lambda_a$-blocks equal to \begin{tikzpicture}
\draw
    (0,0) 
 -- (0.125,0) node {$3$}
-- (0.25,0) 
-- (0.175,-0.5) 
-- (0.075,-0.5) 
 -- cycle;

\draw
    (0.5,0) 
 -- (0.625,0) node {$2$}
-- (0.75,0) 
-- (0.675,-0.5) 
-- (0.575,-0.5) 
 -- cycle;

\draw
    (1,0) node {$1$}
 -- (1.5,0) node {$1$}		
 -- (1.35,-0.5) 
-- (1.25,-0.5) node {$1$}
-- (1.15,-0.5) 
 -- cycle;
\end{tikzpicture} corresponds to a non-unitary $U(4,1)$-module up to level $\mathfrak{p}^+$, then so is $X$.
\end{example}

\section{Fundamental Case} \label{sec-fund}
\begin{definition} \label{def-fund}
    Let $G = U(p,q)$. A $\lambda_a$-datum or a combinatorial $\theta$-stable datum is called {\bf fundamental} if the content of all its neighbouring $\lambda_a$-blocks have differences $\leq 1$. 
    
    An irreducible, Hermitian $(\mathfrak{g},K)$-module $\Pi$ with real infinitesimal character is a {\bf fundamental module} if its corresponding  combinatorial $\theta$-stable datum is fundamental.
\end{definition}
For instance, all irreducible modules with $\lambda_a$-datum
equal to that of Example \ref{eg-barbasch} is not fundamental, since the content of the second last $\gamma$-block ($\frac{-1}{2}$) and the content of the last $\gamma$-block ($-2$) have difference $\frac{3}{2} > 1$.

\smallskip
As mentioned in the introduction, fundamental representations play an indispensable role in the proof of Conjecture \ref{conj-original} and Conjecture \ref{conj-fpp} for $U(p,q)$. Indeed, both conjectures rely on the validity of the following theorem, whose proof with take up the rest of this section:
\begin{theorem} \label{thm-upq}
   Let $\Pi$ be a fundamental module with
   infinitesimal character $\Lambda$. If $\Pi$  is unitary, then $\Lambda$ satisfies $\langle \Lambda, \alpha^{\vee} \rangle \leq 1$ for all simple roots $\alpha \in \Delta^+(\mathfrak{g},\mathfrak{h})$. Otherwise, $\Pi$ is not unitary up to level $\mathfrak{p}$ (Definition \ref{def-upto}). 
\end{theorem}

To begin with, let $\Pi$ be a fundamental module such that there is a `gap' in its infinitesimal character $\Lambda$:
$$\Lambda = (\dots \geq \lambda_i > \lambda_{i+1} \geq \dots), \quad \quad \lambda_i - \lambda_{i+1} > 1.$$ 

Let $\alpha$ and $\omega$ be the largest and smallest coordinates of $\lambda_a$, there are three possibilities for the position of the $>1$ gap:
\begin{itemize}
\item[(a)] $\alpha \geq \lambda_i, \lambda_{i+1} \geq \omega$. In other words, the gap occurs \emph{within $\lambda_a$};
\item[(b)] $\lambda_i > \alpha$;
\item[(b')] $\omega \geq \lambda_{i+1}$. 
\end{itemize}
In Section \ref{subsec-casea}, we prove Theorem \ref{thm-upq} for Case (a). By symmetry, we will only consider for Case (b) in Section \ref{subsec-caseb}.

\subsection{Proof of Theorem \ref{thm-upq} - Case (a)} \label{subsec-casea}
Since $\Pi$ is fundamental, the fact that $\lambda_i - \lambda_{i+1}$ implies that there must be at least one $\gamma$-block in the $\lambda_a$-datum of $X$ satisfying $\lambda_i \geq \gamma \geq \lambda_{i+1}$. Since there are no coordinates in between  $\lambda_i$ and $\lambda_{i+1}$, the $\gamma$-block must be a rectangle or parallelogram of the form

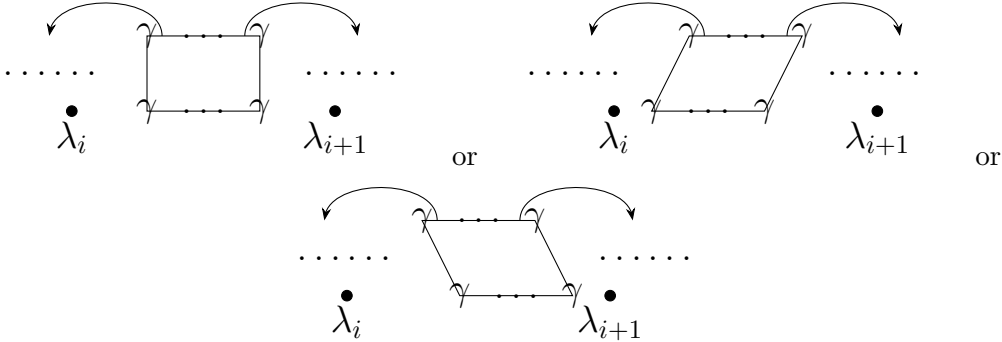
\begin{center}
\begin{tikzpicture}
\draw (-1.25,-0.5) node {\Large $\dots \dots$};
\draw (2.75,-0.5) node {\Large $\dots \dots$};

\draw
    (0,0) node {\Large $\gamma$}
 -- (0.75,0) node {\Large $\ldots$}
 -- (1.5,0) node {\Large $\gamma$}
-- (1.5,-1) node {\Large $\gamma$}
-- (0.75,-1) node {\Large $\ldots$}
-- (0,-1) node {\Large $\gamma$}
 -- cycle;

\draw[arrows = {-Stealth[]}]          (0.2,0)   to [out=90,in=90]node[above]{} (-1.3,0);
\draw[arrows = {-Stealth[]}]          (1.3,0)   to [out=90,in=90]node[above]{} (2.8,0);

\filldraw[black] (-1,-1) circle (2pt) node[below]{\Large $\lambda_i$};
\filldraw[black] (2.5,-1) circle (2pt) node[below]{\Large $\lambda_{i+1}$};

\end{tikzpicture} 
\quad or \quad
\begin{tikzpicture}

\draw (-1.5,-0.5) node {\Large $\dots \dots$};
\draw (2.5,-0.5) node {\Large $\dots \dots$};

\draw
    (0,0) node {\Large $\gamma$}
 -- (0.75,0) node {\Large $\ldots$}
 -- (1.5,0) node {\Large $\gamma$}
-- (1,-1) node {\Large $\gamma$}
-- (0.25,-1) node {\Large $\ldots$}
-- (-0.5,-1) node {\Large $\gamma$}
 -- cycle;

\draw[arrows = {-Stealth[]}]          (0.2,0)   to [out=90,in=90]node[above]{} (-1.3,0);
\draw[arrows = {-Stealth[]}]          (1.3,0)   to [out=90,in=90]node[above]{} (2.8,0);

\filldraw[black] (-1,-1) circle (2pt) node[below]{\Large $\lambda_i$};
\filldraw[black] (2.5,-1) circle (2pt) node[below]{\Large $\lambda_{i+1}$};
\end{tikzpicture}
\quad or \quad
\begin{tikzpicture}
\draw (-1,-0.5) node {\Large $\dots \dots$};
\draw (3,-0.5) node {\Large $\dots \dots$};

\draw
    (0,0) node {\Large $\gamma$}
 -- (0.75,0) node {\Large $\ldots$}
 -- (1.5,0) node {\Large $\gamma$}
-- (2,-1) node {\Large $\gamma$}
-- (1.25,-1) node {\Large $\ldots$}
-- (0.5,-1) node {\Large $\gamma$}
 -- cycle;

\draw[arrows = {-Stealth[]}]          (0.2,0)   to [out=90,in=90]node[above]{} (-1.3,0);
\draw[arrows = {-Stealth[]}]          (1.3,0)   to [out=90,in=90]node[above]{} (2.8,0);

\filldraw[black] (-1,-1) circle (2pt) node[below]{\Large $\lambda_i$};
\filldraw[black] (2.5,-1) circle (2pt) node[below]{\Large $\lambda_{i+1}$};
\end{tikzpicture}
\end{center}
Note that the arrows goes beyond $\lambda_i$ and $\lambda_{i+1}$, that is, the $\nu$-values
attached to the above $\gamma$-blocks must be of the form 
\begin{equation} \label{eq-nubig}
(\nu_1 \geq \dots \geq \nu_r \geq -\nu_r \geq \dots \geq -\nu_1),\quad \quad \nu_r \geq \max\{\lambda_i - \gamma, \gamma - \lambda_{i+1}\} > \frac{1}{2}.
\end{equation}

\begin{proposition} \label{prop-parallel}
Let $G' := U(r,r)$, and $\Theta$ be an irreducible $(\mathfrak{g}',K')$-module corresponding to a single $\lambda_a$-parallelogram $\gamma$-block with $\nu$ of the form \eqref{eq-nubig}. Then $\Theta$ has two lowest $K$-types of highest weights
$$\mu_{+} = (a+1, \dots, a+1\ |\ a, \dots, a), \quad \quad \mu_{-} = (a, \dots, a\ |\ a+1, \dots, a+1),$$
with $a = \gamma - \frac{1}{2}$, and is non-unitary on the $\mathfrak{p}^{\pm}$-level. More precisely, the form is indefinite on
these two pairs of $K$-types with highest weights:
\begin{equation} \label{eq-mupmx}
\{\mu_{+}, (a+1, \dots, a+1, a\ |\ a+1, \dots, a)\}, \quad \quad \{\mu_{-}, (a+1,a, \dots, a\ |\ a+1, \dots, a+1,a)\}.
\end{equation}
\end{proposition}

\begin{proof}
The first statement is immediate from Example \ref{eg-lkts}. As for the second statement,
suppose $\nu$ be of the form \eqref{eq-nubig}. Let 
$$\nu(t) := (\nu_1+t, \dots, \nu_r+t,  - \nu_r-t, \dots, \gamma - \nu_1-t)$$ 
for all $t \geq 0$, and $\Theta(t)$
be the irreducible $(\mathfrak{g}',K')$-module with the same $\lambda_a$-parallelogram $\gamma$-block with $\nu$-coordinates given by $\nu(t)$. Consider the induced module
$$I(t) := Ind_{GL(r,\mathbb{C})N}^{G'}\left(J_r(\gamma + \nu_1+t, \dots, \gamma + \nu_r+t; \gamma - \nu_1-t, \dots, \gamma - \nu_r-t) \boxtimes 1\right).$$
where $J_n(\lambda_L;\lambda_R)$ is the irreducible representation of $GL(n,\mathbb{C})$ with
Zhelobenko parameter $(\lambda_L;\lambda_R)$ (which is also the infinitesimal character of the module).

\medskip
\noindent {\bf Claim 1: For all $t \geq 0$, $I(t)$ has $\Theta(t)$ as the irreducible quotient.} Indeed, the standard module $X(\gamma, \nu(t))$ corresponding to $\Theta(t)$
(given by \cite[Theorem 4.3]{SRV98} for instance) has restricted root system $\Delta(\mathfrak{g}',\mathfrak{a}')$ of Type $C_r$
corresponding to the split non-compact Cartan subalgebra $\mathfrak{a}_0' \cong \mathbb{R}^r$.
Then $\Theta(t)$ is the image of the long intertwining operator 
$$\iota(w_0): X(\gamma, \nu(t)) \longrightarrow X(w_0 \cdot \gamma,w_0\cdot \nu(t)).$$
for the longest Weyl group element $-1 = w_0 \in W(C_r)$. 

We split $w_0 = -1$ into the following parts:
\begin{align*}
&(\nu_1+t,  \nu_2+t, \nu_3 + t, \dots, \nu_r+t)\\
\stackrel{w_{GL}}{\longrightarrow}\  &(\nu_r+t, \dots, \nu_3+t, \nu_2+t, \nu_1 + t)  \stackrel{s_{long}}{\longrightarrow}  (\nu_r+t, \dots, \nu_3 + t, \nu_2+t, -\nu_1 - t) \\
\stackrel{\sigma_1}{\longrightarrow}\ &(-\nu_1-t, \nu_r+t, \dots, \nu_3 + t, \nu_2 + t)  \stackrel{s_{long}}{\longrightarrow}   (-\nu_1-t, \nu_r+t, \dots \nu_3 + t, -\nu_2 - t) \\
\stackrel{\sigma_2}{\longrightarrow}\  &(-\nu_1-t, -\nu_2 - t, \nu_r+t, \dots, \nu_3 + t)  \stackrel{s_{long}}{\longrightarrow}    (-\nu_1-t, -\nu_2 - t, \nu_r+t, \dots, -\nu_3 - t)\\
&\dots \\
\stackrel{\sigma_{r-1}}{\longrightarrow}\  &(-\nu_1-t, -\nu_2-t, \dots, -\nu_{r-1}-t,\nu_r + t)  \stackrel{s_{long}}{\longrightarrow}   (-\nu_1-t, \dots -\nu_{r-1} - t, -\nu_r - t) 
\end{align*}
Since $w_{GL}$ is the long Weyl group element of the Levi subgroup $W(A_{r-1}) \leq W(C_r)$, we have $\mathrm{im}(\iota(w_{GL})) = I(t)$ by induction
in stages. Therefore, $I(t)$ inherits the Hermitian form of the standard module, and has irreducible quotient $\Theta(t)$. 

\medskip
\noindent {\bf Claim 2: $I(t)$ and $\Theta(t)$ has the same multiplicities and signatures up to level $\mathfrak{p}^{\pm}$.} This amounts to showing the intertwining operators 
$\iota(s_{long})$ and $\iota(\sigma_i)$ for $s_{long}$ and $\sigma_i$ appearing above have no kernel up to level $\mathfrak{p}^{\pm}$.

\smallskip
We study $\iota(s_{long})$ first. By induction in stages, it amounts to studying the intertwining operator:
$$Ind_{GL(1,C)N}^{U(1,1)}(J_1(\gamma + \nu_i+t; \gamma - \nu_i-t) \boxtimes 1) \longrightarrow Ind_{GL(1,C)N}^{U(1,1)}(J_1(\gamma - \nu_i-t; \gamma + \nu_i+t) \boxtimes 1).$$
for all $i$. Indeed, the above operator has kernel on the level of $\mathfrak{p}^{\pm}$ if and only if $\nu_i = \frac{1}{2}$, so that the image is a unitary character of $U(1,1)$
(see also \cite[Proposition 7.4]{BJ90b}) . Since all $\nu_i+t > \frac{1}{2}$ by our hypothesis, so $\iota(s_{long})$ has no kernel on the level of $\mathfrak{p}^{\pm}$.

Similarly, $\iota(\sigma_i)$ is a composition of intertwining operators induced from operators of the form
\begin{align*}
&Ind_{GL(1,\mathbb{C}) \times GL(1,\mathbb{C})}^{GL(2,\mathbb{C})}(J_1(\gamma + \nu_i+t; \gamma - \nu_i-t) \boxtimes J_1(\gamma - \nu_j-t; \gamma + \nu_j+t)) \\
\longrightarrow\ &Ind_{GL(1,\mathbb{C}) \times GL(1,\mathbb{C})}^{GL(2,\mathbb{C})}(J_1(\gamma - \nu_j-t; \gamma + \nu_j+t) \boxtimes J_1(\gamma + \nu_i+t; \gamma - \nu_i-t) ).
\end{align*}
The lowest $U(2)$-type of the above module has highest weight $(2\gamma,2\gamma)$. As before, the above operator has kernel on level $\mathfrak{p}$ (i.e. the $U(2)$-type with highest weight $(2\gamma+1,2\gamma-1)$) if and only if $(\gamma + \nu_i +t) - (\gamma-\nu_j -t) = \nu_i + \nu_j +2t = 1$, which is again impossible by our hypothesis, so 
 $\iota(\sigma)$ also has no kernel on the level of $\mathfrak{p}^{\pm}$ for all $i$, and the claim is proved.

\medskip
We are now in the position to prove the proposition. By the classification of irreducible representations of $GL(n,\mathbb{C})$ (for instance \cite[Proposition 12.2]{V86}), all $I(t)$ 
have the same $K'$-type multiplicities. Then Claim 2 and Corollary \ref{cor-jantzen} implies that all $\Theta(t)$ has the same multiplicities and signatures up to level $\mathfrak{p}^{\pm}$. However, when $t >> 0$ is large, the two pairs of $K'$-types in \eqref{eq-mupmx} satisfy the Parthasarathy's inequality in Theorem \ref{thm-partha}. Consequently, $\Theta(t)$ as well as
$\Theta(0) = \Theta$ has indefinite Hermitian forms on these $K'$-types.
\end{proof}

\begin{remark} \label{rmk-parallel}
In fact, if $r$ is odd in Proposition \ref{prop-parallel}, then $\Theta$ has indefinite forms on the lowest $K$-types with highest weights
$\{\mu_+, \mu_-\}$. This follows from \cite[Theorem 11.2]{AvLTV20} --
namely, one can take $x = \mathrm{diag}(I_r, -I_r)$, in \cite[Definition 11.1]{AvLTV20}, so that
$\epsilon(\mu_{\pm}) := \mu_{\pm}(x) = \det^a(-I_r) = (-1)^{ar}$ or $\det^{a+1}(-I_r) = (-1)^{(a+1)r}$, which have different signs if $r$ is odd.
\end{remark}

Now we study the representation corresponding to a single rectangular $\gamma$-block with $\nu$ coordinates
satisfying \eqref{eq-nubig}. Since the group $G' := U(r,r)$ is quasisplit, 
and the lowest $K'$-type $\delta$ is trivial on the semisimple part, one can apply \cite{BJ90a} 
(a more general result is given in \cite{B10}), or simply copy the proof of Proposition \ref{prop-parallel} above to conclude 
that the Hermitian form is indefinite on the level of $\mathfrak{p}^{\pm}$:
\begin{equation} \label{eq-pplus1}
\mu(\delta) = (a,\dots, a | a, \dots, a), \quad (a+1, a, \dots, a | a, \dots, a, a-1)
\end{equation}
and
\begin{equation} \label{eq-pminus1}
\mu(\delta) = (a,\dots, a | a, \dots, a), \quad (a, \dots, a, a-1 |a+1, a, \dots, a)
\end{equation} 
with $a = \gamma$.

We now go back to studying the 3 cases in the beginning of Section \ref{subsec-casea}. The last two parallelogram cases,
i.e. 
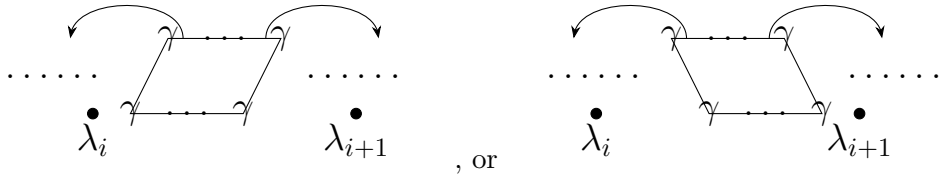
\begin{center}
\begin{tikzpicture}
\draw (-1.5,-0.5) node {\Large $\dots \dots$};
\draw (2.5,-0.5) node {\Large $\dots \dots$};

\draw
    (0,0) node {\Large $\gamma$}
 -- (0.75,0) node {\Large $\ldots$}
 -- (1.5,0) node {\Large $\gamma$}
-- (1,-1) node {\Large $\gamma$}
-- (0.25,-1) node {\Large $\ldots$}
-- (-0.5,-1) node {\Large $\gamma$}
 -- cycle;

\draw[arrows = {-Stealth[]}]          (0.2,0)   to [out=90,in=90]node[above]{} (-1.3,0);
\draw[arrows = {-Stealth[]}]          (1.3,0)   to [out=90,in=90]node[above]{} (2.8,0);

\filldraw[black] (-1,-1) circle (2pt) node[below]{\Large $\lambda_i$};
\filldraw[black] (2.5,-1) circle (2pt) node[below]{\Large $\lambda_{i+1}$};
\end{tikzpicture}
\quad, or \quad
\begin{tikzpicture}
\draw (-1,-0.5) node {\Large $\dots \dots$};
\draw (3,-0.5) node {\Large $\dots \dots$};

\draw
    (0,0) node {\Large $\gamma$}
 -- (0.75,0) node {\Large $\ldots$}
 -- (1.5,0) node {\Large $\gamma$}
-- (2,-1) node {\Large $\gamma$}
-- (1.25,-1) node {\Large $\ldots$}
-- (0.5,-1) node {\Large $\gamma$}
 -- cycle;

\draw[arrows = {-Stealth[]}]          (0.2,0)   to [out=90,in=90]node[above]{} (-1.3,0);
\draw[arrows = {-Stealth[]}]          (1.3,0)   to [out=90,in=90]node[above]{} (2.8,0);

\filldraw[black] (-1,-1) circle (2pt) node[below]{\Large $\lambda_i$};
\filldraw[black] (2.5,-1) circle (2pt) node[below]{\Large $\lambda_{i+1}$};
\end{tikzpicture}
\end{center}
are indeed isomorphic by our choice of $\nu$ in \eqref{eq-nubig}
and Example \ref{eg-lkts}. By Proposition \ref{prop-bottom}, the left $\gamma$-block is $\mathfrak{p}^+$-bottom layer (Definition \ref{def-bottom}) and the right $\gamma$-block is
$\mathfrak{p}^-$-bottom layer. By applying Corollary \ref{cor-bottom} to 
one of the two pairs of indefinite $K'$-types in Equation \ref{eq-mupmx} of Proposition \ref{prop-parallel}, 
one concludes that $\Pi$ is non-unitary on the level of $\mathfrak{p}^{\pm}$.

\medskip
In the rectangular $\gamma$-block case, suppose the combinatorial $\theta$-stable data of $\Pi$ looks like:

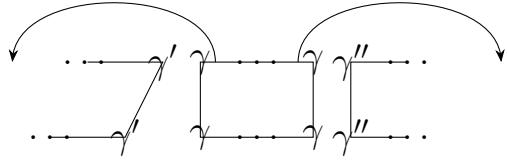
\begin{center}
\begin{tikzpicture}
\draw
    (-1.5,0) node {\Large $\dots$}
 -- (-0.5,0) node {\Large $\gamma'$}
 -- (-1,-1) node {\Large $\gamma'$}
-- (-2,-1) node {\Large $\ldots$};

\draw
    (0,0) node {\Large $\gamma$}
 -- (0.75,0) node {\Large $\ldots$}
 -- (1.5,0) node {\Large $\gamma$}
 -- (1.5,-1) node {\Large $\gamma$}
-- (0.75,-1) node {\Large $\ldots$}
-- (0,-1) node {\Large $\gamma$}
 -- cycle;

\draw[arrows = {-Stealth[]}]          (0.2,0)   to [out=90,in=90]node[above]{} (-2.5,0);
\draw[arrows = {-Stealth[]}]          (1.3,0)   to [out=90,in=90]node[above]{} (4,0);

\draw
    (2.75,0) node {\Large $\ldots$}
-- (2,0) node {\Large $\gamma''$}
-- (2,-1) node {\Large $\gamma''$}
-- (2.75,-1) node {\Large $\ldots$};



\end{tikzpicture} 
\end{center}
Then $\Pi$ is $\mathfrak{p}^-$-bottom layer on the 
$\gamma$-block by Proposition \ref{prop-bottom}. On the other hand, since the 
$\gamma$-block corresponds to a Hermitian representation whose signatures are
indefinite on the level of $\mathfrak{p}^-$ by \eqref{eq-pminus1}. By Corollary \ref{cor-bottom}, 
one concludes that $\Pi$ is non-unitary on the level of $\mathfrak{p}^-$.

As a consequence, the only case where the Corollary \ref{cor-bottom} fails to detect non-unitarity 
of $\Pi$ is when the $\lambda_a$-blocks are of the form
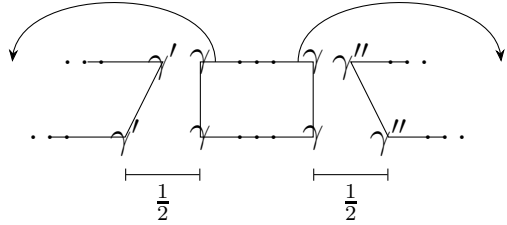
\begin{center}
\begin{tikzpicture}
\draw
    (-1.5,0) node {\Large $\dots$}
 -- (-0.5,0) node {\Large $\gamma'$}
 -- (-1,-1) node {\Large $\gamma'$}
-- (-2,-1) node {\Large $\ldots$};

\draw [|-|](-1,-1.5) -- node[below] {$\frac{1}{2}$} (0,-1.5);
\draw [|-|](1.5,-1.5) -- node[below] {$\frac{1}{2}$} (2.5,-1.5);

\draw
    (0,0) node {\Large $\gamma$}
 -- (0.75,0) node {\Large $\ldots$}
 -- (1.5,0) node {\Large $\gamma$}
 -- (1.5,-1) node {\Large $\gamma$}
-- (0.75,-1) node {\Large $\ldots$}
-- (0,-1) node {\Large $\gamma$}
 -- cycle;

\draw[arrows = {-Stealth[]}]          (0.2,0)   to [out=90,in=90]node[above]{} (-2.5,0);
\draw[arrows = {-Stealth[]}]          (1.3,0)   to [out=90,in=90]node[above]{} (4,0);

\draw
    (2.75,0) node {\Large $\ldots$}
-- (2,0) node {\Large $\gamma''$}
-- (2.5,-1) node {\Large $\gamma''$}
-- (3.25,-1) node {\Large $\ldots$};



\end{tikzpicture} 
\end{center}
(or flipped upside down), where $\gamma' - \gamma = \gamma - \gamma'' = \frac{1}{2}$. Since $\lambda_i - \lambda_{i+1} > 1$,
at least one of $\gamma'$ or $\gamma''$ must lie in the open interval $(\lambda_{i+1}, \lambda_i)$. Without loss of generality, say the $\lambda_a$-block with content $\gamma'$ lies in the interval. Then the block must be a parallelogram with non-zero $\nu$-entries:
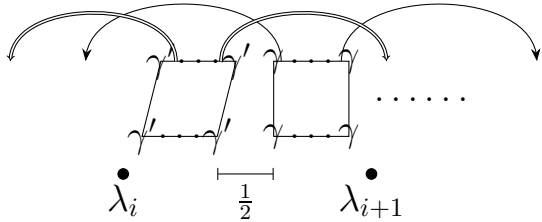
\begin{center}
\begin{tikzpicture}
\draw
    (-2,0) node {\Large $\gamma'$}
-- (-1.5,0) node {\Large $\ldots$}
 -- (-1,0) node {\Large $\gamma'$}
 -- (-1.25,-1) node {\Large $\gamma'$}
-- (-1.75,-1) node {\Large $\ldots$}
-- (-2.25,-1) node {\Large $\gamma'$}
-- cycle;

\filldraw[black] (-2.5,-1.5) circle (2pt) node[below]{\Large $\lambda_i$};
\draw [|-|](-1.25,-1.5) -- node[below] {$\frac{1}{2}$} (-0.5,-1.5);
\filldraw[black] (0.8,-1.5) circle (2pt) node[below]{\Large $\lambda_{i+1}$};

\draw
    (-0.5,0) node {\Large $\gamma$}
 -- (0,0) node {\Large $\ldots$}
 -- (0.5,0) node {\Large $\gamma$}
 -- (0.5,-1) node {\Large $\gamma$}
-- (0,-1) node {\Large $\ldots$}
-- (-0.5,-1) node {\Large $\gamma$}
 -- cycle;

\draw[arrows = {-Stealth[]}]          (-0.4,0)   to [out=90,in=90]node[above]{} (-3,0);
\draw[arrows = {-Stealth[]}]          (0.4,0)   to [out=90,in=90]node[above]{} (3,0);

\draw (1.5,-0.5) node {\Large $\dots \dots$};

\draw[-{Implies},double]          (-1.8,0)   to [out=90,in=90]node[above]{} (-4,0);
\draw[-{Implies},double]          (-1.2,0)   to [out=90,in=90]node[above]{} (1,0);

\end{tikzpicture} 
\end{center}
Then we are reduced to the case of a parallelogram $\lambda_a$-block with content $\gamma'$,
and hence the result follows.

\begin{example}
Let $G = U(4,3)$ and $X$ corresponds to the combinatorial $\theta$-stable data:
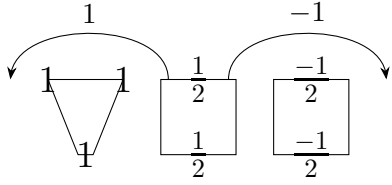
\begin{center}
\begin{tikzpicture}
\draw
    (0,0) node {\Large $1$}
 -- (1,0) node {\Large $1$}
-- (0.6, -1)
-- (0.5,-1) node {\Large $1$}
-- (0.4, -1)
 -- cycle;

\draw[arrows = {-Stealth[]}]          (1.6,0)   to [out=90,in=90]node[above]{$1$} (-0.5,0);
\draw[arrows = {-Stealth[]}]          (2.4,0)   to [out=90,in=90]node[above]{$-1$} (4.5,0);

\draw
    (1.5,0) 
		-- (2,0) node {\Large $\frac{1}{2}$}
 -- (2.5,0) 
-- (2.5,-1) 
-- (2,-1) node {\Large $\frac{1}{2}$}
-- (1.5,-1) 
 -- cycle;

\draw
    (3,0) 
		-- (3.5,0) node {\Large $\frac{-1}{2}$}
 -- (4,0) 
-- (4,-1) 
-- (3.5,-1) node {\Large $\frac{-1}{2}$}
-- (3,-1) 
 -- cycle;
\end{tikzpicture}
\end{center}
Therefore, $X$ has infinitesimal character 
$$\Lambda = (1,1,1,\frac{1}{2}+1,\frac{1}{2}-1,\frac{-1}{2},\frac{-1}{2}) = (\frac{3}{2},1,1,1,\frac{-1}{2},\frac{-1}{2},\frac{-1}{2}),$$ 
and lowest $K$-type $\delta$ of highest weight:
\begin{align*}
\mu = &\Big(1,1,{\bf \frac{1}{2}},\frac{-1}{2} \Big| 1,{\bf \frac{1}{2}}, \frac{-1}{2}\Big) - \rho(\mathfrak{u}) + 2\rho(\mathfrak{u} \cap \mathfrak{p})\\ 
= &\Big(1,1,{\bf \frac{1}{2}},\frac{-1}{2} \Big| 1,{\bf \frac{1}{2}}, \frac{-1}{2}\Big) - \Big(2,2,{\bf \frac{-1}{2}},\frac{-5}{2} \Big| 2,{\bf \frac{-1}{2}}, \frac{5}{2}\Big) + (2,2,{\bf 0},-2 | 2,{\bf-1},-3) \\
= &(1,1,{\bf 1},0 | 1,{\bf 0},-1).
\end{align*}
The highlighted coordinates are those coming from the $\lambda_a$-block with content $\frac{1}{2}$. 
By \eqref{eq-pminus1}, the $K$-type $(1,1,{\bf 1-1},0 | 1,{\bf 0+1},-1) = (1,1,0,0 | 1,1,-1)$ is bottom layer (this can
also be read off directly from the combinatorial $\theta$-stable data), and has
different signature as $\mu$. 

\medskip
To verify this, we implement this representation on \texttt{atlas}:
\begin{verbatim}
atlas> set G = U(4,3)
atlas> set p = parameter(G,419,[2,1,1,1,-1,0,-1]/1,[1,0,0,0,-1,0,0]/1)
atlas> infinitesimal_character(p)
Value: [  3,  2,  2,  2, -1, -1, -1 ]/2
atlas> print_branch_irr_long(p,KGB(G,34),height(p))
m  x   lambda                    hw                      dim  height
1  95  [1, 1, 1, 1, 0, 0, -1]/1  [1, 1, 1, 0, 1, 0, -1]  32   16
\end{verbatim}
This says \texttt{p} has the correct infinitesimal character $\Lambda$ and lowest $K$-type $\mu$ (by looking
at the \texttt{hw} column). To look at the signatures of the $K$-types, we have
\begin{verbatim}
atlas> print_sig_irr_long(p,KGB(G,34),height(p)+2)
sig  x  lambda                    hw                      dim  height
1   95  [1, 1, 1, 1, 0, 0, -1]/1  [1, 1, 1, 0, 1, 0, -1]  32   16
1   58  [1, 1, 1, 1, 0, 0, -1]/1  [1, 1, 1, 1, 1, 0, -2]  15   17
1   55  [1, 1, 1, 1, 0, 0, -1]/1  [1, 1, 1, -1, 1, 0, 0]  30   17
s   35  [1, 1, 1, 1, 0, 0, -1]/1  [1, 1, 0, 0, 1, 1, -1]  36   17
s   0   [1, 1, 1, 1, 0, 0, -1]/1  [1, 1, 0, -1, 1, 1, 0]  60   18
\end{verbatim}
The first and fourth row have different \texttt{sig} values \texttt{1} and \texttt{s}. 
This implies that the form is indefinite on these two $K$-types, verifying our proof in this case.
\end{example}

\subsection{Proof of Theorem \ref{thm-upq} - Case (b)} \label{subsec-caseb}
We now deal with the case when there is $\lambda_i > \alpha$. This happens when there is a $\lambda_a$-block of content $\gamma$ having at least one $\nu$-coordinate surpassing the other $\lambda_a$-blocks, for instance:
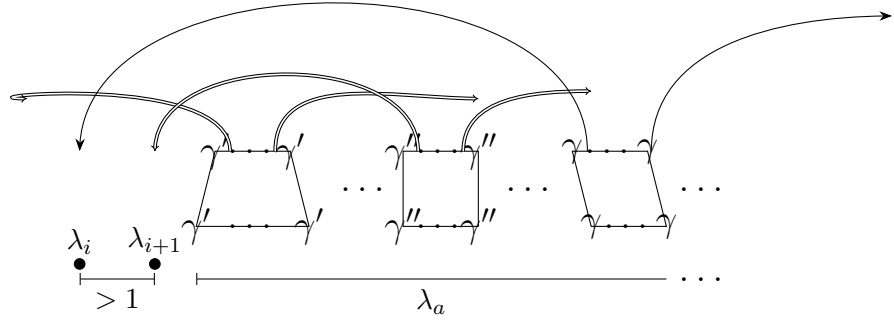
\begin{center}
\begin{tikzpicture}
\draw

    (-2,0) node {\Large $\gamma'$}
 -- (-1.5,0) node {\Large $\dots$}
-- (-1,0) node {\Large $\gamma'$}
 -- (-0.75,-1)  node {\Large $\gamma'$}
-- (-1.5,-1) node {\Large $\dots$}
-- (-2.25,-1) node {\Large $\gamma'$}
 -- cycle;

\draw[-{Implies},double]          (-1.8,0)   to [out=90,in=180]node[above]{} (-4.5,0.7);
\draw[-{Implies},double]         (-1.2,0)   to [out=90,in=180]node[above]{} (1.5,0.7);

\draw (0,-0.5) node {\Large $\dots$};

\draw
    (0.5,0) node {\Large $\gamma''$}
 -- (1,0) node {\Large $\dots$}
-- (1.5,0) node {\Large $\gamma''$}
 -- (1.5,-1)  node {\Large $\gamma''$}
-- (1,-1) node {\Large $\dots$}
-- (0.5,-1) node {\Large $\gamma''$}
 -- cycle;

\draw[-{Implies},double]          (0.7,0)   to [out=90,in=90]node[above]{} (-2.8,0);
\draw[-{Implies},double]         (1.3,0)   to [out=90,in=180]node[above]{} (3,.8);

\draw (2.2,-0.5) node {\Large $\dots$};

\draw
    (2.75,0) node {\Large $\gamma$}
 -- (3.25,0) node {\Large $\dots$}
-- (3.75,0) node {\Large $\gamma$}
 -- (4,-1)  node {\Large $\gamma$}
-- (3.5,-1) node {\Large $\dots$}
-- (3,-1) node {\Large $\gamma$}
 -- cycle;

\draw[arrows = {-Stealth[]}]          (2.95,0)   to [out=90,in=90]node[above]{} (-3.8,0);
\draw[arrows = {-Stealth[]}]            (3.8,0)   to [out=90,in=180]node[above]{} (7,1.8);

\draw (4.5,-0.5) node {\Large $\dots$};

\filldraw[black] (-3.8,-1.5) circle (2pt) node[above]{$\lambda_i$};
\filldraw[black] (-2.8,-1.5) circle (2pt) node[above]{$\lambda_{i+1}$};
\draw [|-|](-3.8,-1.7) -- node[below] {$>1$} (-2.8,-1.7);
\draw [|-](-2.25,-1.7) -- node[below] {$\lambda_a$} (4,-1.7);
\draw (4.5,-1.7) node {\Large $\dots$};

\end{tikzpicture}
\end{center}
From now on, we call the $\lambda_a$-block with content $\gamma$ of the above form a {\bf $\lambda$-large block}.
By switching $U(p,q)$ to $U(q,p)$ if necessary, one can assume the $\lambda$-large block is of one of the forms:
\begin{equation} \label{eq-blocks}
 \begin{tikzpicture} \draw
    (2,0) node {\Large $\gamma$}
 -- (2.75,0) node {\Large $\dots$}
-- (3.5,0) node {\Large $\gamma$}
 -- (3.25,-1)  node {\Large $\gamma$}
-- (2.75,-1) node {\Large $\dots$}
-- (2.25,-1) node {\Large $\gamma$}
 -- cycle; \end{tikzpicture} \quad \quad \quad
\begin{tikzpicture} \draw
    (2,0) node {\Large $\gamma$}
 -- (2.5,0) node {\Large $\dots$}
-- (3,0) node {\Large $\gamma$}
 -- (3,-1)  node {\Large $\gamma$}
-- (2.5,-1) node {\Large $\dots$}
-- (2,-1) node {\Large $\gamma$}
 -- cycle; \end{tikzpicture}.
\quad \quad \quad
\begin{tikzpicture} \draw
    (1.75,0) node {\Large $\gamma$}
 -- (2.25,0) node {\Large $\dots$}
-- (2.75,0) node {\Large $\gamma$}
 -- (3,-1)  node {\Large $\gamma$}
-- (2.5,-1) node {\Large $\dots$}
-- (2,-1) node {\Large $\gamma$}
 -- cycle; \end{tikzpicture}
\end{equation}

\begin{definition} \label{def-semisph}
Let $\Pi$ be a fundamental representation, whose combinatorial $\theta$-stable datum contains a $\lambda$-large $\gamma$-block
of the form \eqref{eq-blocks}. Then the {\bf semi-spherical 
component} corresponding to the $\lambda$-large $\gamma$-block 
is the longest chain of $\lambda_a$-blocks corresponding to $\Pi$ satisfying:
\begin{itemize}
\item The $\gamma$-block is the rightmost block; and
\item For two neighboring $\lambda_a$-blocks, they must be of one of the following forms:
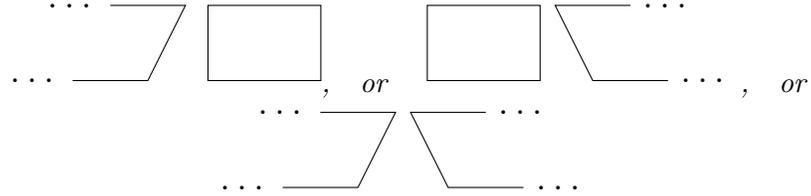
\begin{center}
\begin{tikzpicture}
\draw
    (1,0) 
-- (2,0)
 -- (1.5,-1)  
-- (0.5,-1);

\draw (0.5,0) node {\Large $\dots$};
\draw (0,-1) node {\Large $\dots$};

\draw
    (3.8,0) 
 -- (2.3,0) 
-- (2.3,-1)
 -- (3.8,-1)
-- cycle;
\end{tikzpicture},\quad or \quad
\begin{tikzpicture}
\draw
    (3,0) 
-- (2,0)
 -- (2.5,-1)  
-- (3.5,-1);

\draw (3.5,0) node {\Large $\dots$};
\draw (4,-1) node {\Large $\dots$};

\draw
    (1.8,0) 
 -- (0.3,0) 
-- (0.3,-1)
 -- (1.8,-1)
-- cycle;
\end{tikzpicture},\quad or \quad
\begin{tikzpicture}
\draw
    (3,0) 
-- (2,0)
 -- (2.5,-1)  
-- (3.5,-1);

\draw (3.5,0) node {\Large $\dots$};
\draw (4,-1) node {\Large $\dots$};

\draw
    (0.8,0) 
-- (1.8,0)
 -- (1.3,-1)  
-- (0.3,-1);

\draw (0.3,0) node {\Large $\dots$};
\draw (-0.2,-1) node {\Large $\dots$};
\end{tikzpicture}
\end{center}
\end{itemize}
(see Example \ref{eg-gammabig} for one such example). In other words, suppose the highest weight of the 
a lowest $K$-type $\delta$ of $\Pi$ is of the form
$$\mu = \mu(\delta) = (\cdots, \fbox{a, \dots, a}, \cdots | \cdots, \fbox{b, \dots, b}, \cdots)$$
where the boxed coordinates are the $\gamma$-component (Definition \ref{def-bottom}(a)) of the $\lambda$-large block,
then the semi-spherical component of this block corresponds to to all `$a$' coordinates of $\mu$:
$$\mu = (\cdots > \underline{a, \dots, a}, \fbox{a, \dots, a} \geq \cdots | \cdots \geq \underline{b',\dots, b',\cdots, b'', \dots, b''}, \fbox{b, \dots, b} > \cdots)$$

We say $\Sigma$ is {\bf semi-spherical} if its combinatorial $\theta$-stable datum contains precisely a single semi-spherical component. In other words, a lowest $K$-type of $\Sigma$ has highest weight of the form $(a,\dots, a\ |\ b',\dots,b', \cdots, b, \dots,b)$.
\end{definition}

\begin{remark}
In \cite{KS83}, Knapp and Speh defined the notion of {\bf basic cases} for all real reductive groups of equal rank, and they were used extensively to study the unitary dual of $U(n,2)$. In the case of $U(p,q)$, one can check by directly applying \cite[Theorem 1.1]{KS83} that all basic cases are semi-spherical.

\end{remark}

\begin{theorem} \label{thm-semipm}
Let $\Sigma$ be a semi-spherical module, with infinitesimal character satisfying Case (b) of Theorem \ref{thm-upq}. Then the Hermitian form of $\Sigma$
is non-unitary on the level of $\mathfrak{p}^{+}$. 

\end{theorem}

\begin{proof} Let $\Sigma$ be a semi-spherical module
whose combinatorial $\theta$-stable datum contains a $\lambda$-large $\gamma$-block.
One can reduce the proof to the case where {\bf the $> 1$ gaps are solely contributed
to the rightmost $\lambda$-large $\gamma$-block}, for instance:
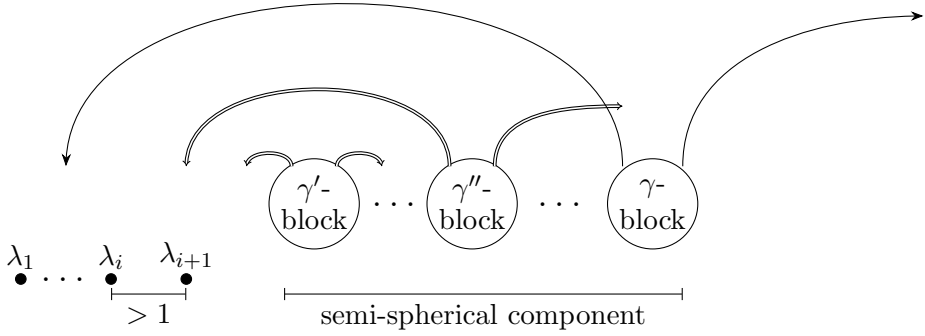
\begin{center}
\begin{tikzpicture}
\draw (-1.1,-0.3) node {$\gamma'$-};
\draw (-1.1,-0.7) node {block};
\draw (-1.1,-0.5) circle (0.6);

\draw[-{Implies},double]          (-1.4,0)   to [out=90,in=90]node[above]{} (-2,0);
\draw[-{Implies},double]         (-0.8,0)   to [out=90,in=90]node[above]{} (-0.2,0);

\draw (0,-0.5) node {\Large $\dots$};

\draw (1,-0.3) node {$\gamma''$-};
\draw (1,-0.7) node {block};
\draw (1,-0.5) circle (0.6);

\draw[-{Implies},double]          (0.7,0)   to [out=90,in=90]node[above]{} (-2.8,0);
\draw[-{Implies},double]         (1.3,0)   to [out=90,in=180]node[above]{} (3,.8);

\draw (2.2,-0.5) node {\Large $\dots$};

\draw (3.4,-0.3) node {$\gamma$-};
\draw (3.4,-0.7) node {block};
\draw (3.4,-0.5) circle (0.6);

\draw[arrows = {-Stealth[]}]          (3,0)   to [out=90,in=90]node[above]{} (-4.4,0);
\draw[arrows = {-Stealth[]}]            (3.8,0)   to [out=90,in=180]node[above]{} (7,2);

\filldraw[black] (-5,-1.5) circle (2pt) node[above]{$\lambda_1$};
\draw (-4.4,-1.5) node {\Large $\dots$};
\filldraw[black] (-3.8,-1.5) circle (2pt) node[above]{$\lambda_i$};
\filldraw[black] (-2.8,-1.5) circle (2pt) node[above]{$\lambda_{i+1}$};
\draw [|-|](-3.8,-1.7) -- node[below] {$>1$} (-2.8,-1.7);
\draw [|-|](-1.5,-1.7) -- node[below] {semi-spherical component} (3.8,-1.7);

\end{tikzpicture}
\end{center}
or the irreducible module studied in Example \ref{eg-gammabig} below. In other words, one has $\Lambda = (\lambda_1 \geq \dots \geq \lambda_i > \lambda_{i+1} \geq \dots)$ with
\begin{center}
$\lambda_i - \lambda_{i+1} > 1$ and $\lambda_{x} = \gamma + \nu_{x}$ for $1\leq x \leq i$.
\end{center}
Indeed, suppose that there is a $\gamma''$-block on the left of $\gamma$-block, with some $(\gamma''+\nu'')$-coordinates 
`surpassing' $\lambda_i$ in the above diagram, then one can first prove Theorem \ref{thm-semipm} holds for the semi-spherical sub-component with $\gamma''$-block being the rightmost block, and then the general case follows from noting that the $L \cap K$-types in Theorem \ref{thm-semipm} corresponding to this sub-component are $\mathfrak{p}^+$-bottom layer (c.f. Remark \ref{rmk-blocks}).

\medskip
Let $\Sigma'$ be the irreducible representation of $U(p-i,q-i)$ corresponding to the combinatorial
data of $\Sigma$ with $(\gamma + \nu_1, \dots, \gamma + \nu_i, \gamma - \nu_i, \dots,  \gamma - \nu_1)$ removed.
Then $\Sigma$ is the irreducible quotient of the real parabolically induced module:
$$Ind_{MAN}^{G}\left(J_i(\gamma+\nu_1, \dots, \gamma + \nu_i ;\gamma - \nu_1, \dots, \gamma - \nu_i) \boxtimes \Sigma' \boxtimes 1\right)$$
with Levi subgroup $L := MA = GL(i,\mathbb{C}) \times U(p-i,q-i)$. Let 
$$I(t) := Ind_{MAN}^{G}\left(J_i(\gamma+\nu_1+t, \dots, \gamma + \nu_i+t ;\gamma - \nu_1-t, \dots, \gamma - \nu_i-t) \boxtimes \Sigma'\boxtimes 1 \right)$$
and $J(t)$ be its irreducible quotient (so that $J(0) = \Sigma$). 
As in Claim 2 in the proof of Proposition \ref{prop-parallel}, we {\it claim} the following holds:
\begin{center}
\fbox{$I(t)$ and $J(t)$ have the same $K$-type multiplicities and signatures up to level $\mathfrak{p}^+$. }
\end{center}
The 
(non-reduced) root system of $U(p,q)$ is of Type $BC_q$, and $I(t) = \mathrm{im}(\iota(w_L))$ is the image of the long intertwining operator of $L$, where $w_L$ is the longest 
Weyl subgroup element in $W(A_{i-1}) \times W(BC_{q-i})  \leq W(BC_q)$), with:
\[(\nu_1+t, \dots, \nu_i+t, \nu_{i+1}, \dots, \nu_q) \xrightarrow{w_L} (\nu_i+t, \dots, \nu_1+t, -\nu_{i+1}, \dots, -\nu_q).\]

Now consider the long intertwining operator $\iota(w_0) = \iota(w') \circ \iota(w_L)$ of $G$, where $\iota(w')$ is split into the following the Weyl group actions:
\begin{itemize}
\item[(i)] $(\dots, \nu_x+t, -\nu_w, \dots) \mapsto (\dots, -\nu_w, \nu_x+t, \dots)$;
\item[(ii)] $(\dots, \nu_x+t) \mapsto (\dots, -\nu_x-t)$;
\item[(iii)] $(\dots, -\nu_w, -\nu_x-t, \dots) \mapsto (\dots, -\nu_x-t, -\nu_w, \dots)$; and
\item[(iv)] $(\dots, \nu_v+t, -\nu_x-t, \dots) \mapsto (\dots, -\nu_x-t, \nu_v+t, \dots)$
\end{itemize}
for $x$, $v$, $w$ satisfying $1 \leq x < v \leq i < w \leq q$. 

\medskip
So the claim above follows if one can show the intertwining operators (i) - (iv) have no kernel up to level $\mathfrak{p}^{\pm}$. The proofs for (i), (iii), (iv) are similar to that in the proof of Proposition \ref{prop-parallel}: for instance, (iii) corresponds to the intertwining operator
\begin{align*}
&Ind_{GL(1,\mathbb{C}) \times GL(1,\mathbb{C})}^{GL(2,\mathbb{C})}(J_1(\gamma'' - \nu_w; \gamma'' + \nu_w) \boxtimes J_1(\gamma - \nu_x-t; \gamma+ \nu_x+t) ) \\
\longrightarrow\ &Ind_{GL(1,\mathbb{C}) \times GL(1,\mathbb{C})}^{GL(2,\mathbb{C})}(J_1(\gamma - \nu_x-t; \gamma + \nu_x+t) \boxtimes J_1(\gamma'' - \nu_w; \gamma'' + \nu_w)).
\end{align*}
The lowest $U(2)$-type of these modules has highest weight $(2\gamma'',2\gamma)$, since $\gamma'' \geq \gamma$ by the definition of semi-spherical blocks.
The kernel of the above intertwining operator does not contain the level $\mathfrak{p}$ $U(2)$-type with highest weight $(2\gamma''+1,2\gamma -1)$. This is due to the fact that 
$$\begin{cases}
|(\gamma'' + \nu_w) - (\gamma+ \nu_x+t)| &> 1\\
|(\gamma'' - \nu_w) - (\gamma - \nu_x-t)| &> 1 
\end{cases}$$
by our hypothesis in Case (b) (recall $\gamma + \nu + t = \lambda_u + t$ and $\gamma'' + \nu_w = \lambda_x$ for some $x > i$), and the character theory of $(\mathfrak{gl}(2,\mathbb{C}),U(2))$-modules. More precisely, if the kernel is non-zero, then a lowest $K$-type of its kernel must be of the form
$$\kappa := (a+w,a,\dots,a,a-x\ |\ b'+y',b',\dots, b';\ \cdots;\ b''+y'', b'', \dots, b'';\ b+y, b, \dots,b, b-z)$$
with $w+(y'+y''+\dots+y) = x+z \geq 2$, $x, z \geq 1$ and at most one $y'$, $y''$, $\dots$, $y$ is nonzero. Then it is easy
to check that $\|\kappa + 2\rho(\mathfrak{k})\|$ is greater than those listed in Theorem \ref{thm-semipm}. 

\medskip
As for (ii), one studies the intertwining operator:
\begin{equation} \label{eq-upq1}
\begin{aligned}
&Ind_{GL(1,\mathbb{C}) \times U(p-q,0)}^{U(p-q+1,1)}(J_1(\gamma + \nu_x+t; \gamma - \nu_x+t) \boxtimes F_{\tau}) \\
\longrightarrow\ &Ind_{GL(1,\mathbb{C}) \times U(p-q,0)}^{U(p-q+1,1)}(J_1(\gamma - \nu_x-t; \gamma + \nu_x+t) \boxtimes F_{\tau}),
\end{aligned}
\end{equation}
where $F_{\tau}$ is the irreducible representation of $U(p-q,0)$ with infinitesimal character $\tau = (\tau_1 > \tau_2 > \dots > \tau_{p-q})$ whose coordinates come from the contents of all trapezoidal blocks of the semi-spherical component. 

The infinitesimal character of the modules in \eqref{eq-upq1} is equal to:
\begin{equation} \label{eq-infl1}
(\gamma + \nu_x + t > \tau_1 > \tau_2 > \dots > \tau_{p-q} > \gamma - \nu_x - t), 
\end{equation}
where $(\gamma + \nu_x + t) - \tau_1 > 1$ by our hypothesis of Case (b) again.

We {\it claim} that the kernel of \eqref{eq-upq1} has no kernel up to $\mathfrak{p}^{+}$ (the claim and its proof also holds for $\mathfrak{p}^{-}$). Indeed, the image of \eqref{eq-upq1}
is precisely the irreducible representation of $U(p-q+1,1)$ with combinatorial $\theta$-stable data:
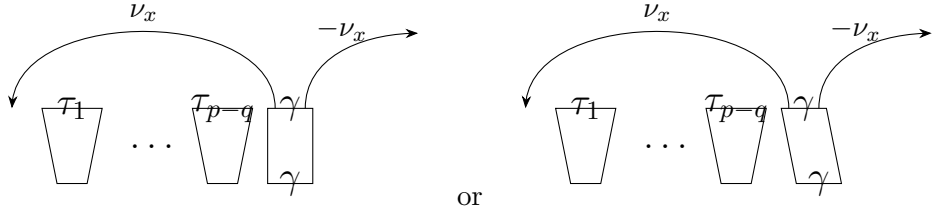
\begin{center}
\begin{tikzpicture}
\draw
    (0,0) 
 -- (0.4,0) node {\Large $\tau_1$}
-- (0.8,0) 
-- (0.6,-1) 
-- (0.2,-1) 
 -- cycle;

\draw (1.5,-0.5) node {\Large $\dots$};

\draw
    (2,0) 
 -- (2.4,0) node {\Large $\tau_{p-q}$}		
 -- (2.8,0) 
 -- (2.6,-1) 
-- (2.2,-1) 
 -- cycle;

\draw
    (3,0) 
 -- (3.3,0) node {\Large $\gamma$}		
 -- (3.6,0) 
 -- (3.6,-1) 
-- (3.3,-1) node {\Large $\gamma$}
-- (3,-1) 
 -- cycle;

\draw[arrows = {-Stealth[]}]          (3.1,0)   to [out=90,in=90]node[above]{$\nu_x$} (-0.4,0);
\draw[arrows = {-Stealth[]}]          (3.5,0)   to [out=90,in=180]node[above]{$-\nu_x$} (5,1);
\end{tikzpicture} 
\quad or  \quad
\begin{tikzpicture}
\draw
    (0,0) 
 -- (0.4,0) node {\Large $\tau_1$}
-- (0.8,0) 
-- (0.6,-1) 
-- (0.2,-1) 
 -- cycle;

\draw (1.5,-0.5) node {\Large $\dots$};

\draw
    (2,0) 
 -- (2.4,0) node {\Large $\tau_{p-q}$}		
 -- (2.8,0) 
 -- (2.6,-1) 
-- (2.2,-1) 
 -- cycle;

\draw
    (3,0) 
 -- (3.3,0) node {\Large $\gamma$}		
 -- (3.6,0) 
 -- (3.8,-1) 
-- (3.5,-1) node {\Large $\gamma$}
-- (3.2,-1) 
 -- cycle;

\draw[arrows = {-Stealth[]}]          (3.1,0)   to [out=90,in=90]node[above]{$\nu_x$} (-0.4,0);
\draw[arrows = {-Stealth[]}]          (3.5,0)   to [out=90,in=180]node[above]{$-\nu_x$} (5,1);
\end{tikzpicture} 
\end{center}
depending on the parity of $2\gamma$. 

Also, by Proposition \ref{prop-khat}, the $\lambda_a$-datum corresponding to the unique $U(p-q+1) \times U(1)$-type in
\begin{center}
(LKT of the modules in \eqref{eq-upq1}) $\otimes\ \mathfrak{p}^+$ 
\end{center}
must be one of the following forms:
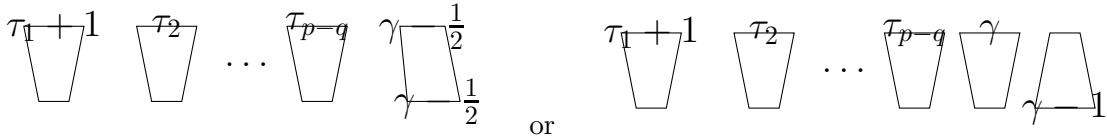
\begin{center}
\begin{tikzpicture}
\draw
    (-1.5,0) 
 -- (-1.1,0) node {\Large $\tau_1+1$}
-- (-0.7,0) 
-- (-0.9,-1) 
-- (-1.3,-1) 
 -- cycle;

\draw
    (0,0) 
 -- (0.4,0) node {\Large $\tau_2$}
-- (0.8,0) 
-- (0.6,-1) 
-- (0.2,-1) 
 -- cycle;

\draw (1.5,-0.5) node {\Large $\dots$};

\draw
    (2,0) 
 -- (2.4,0) node {\Large $\tau_{p-q}$}		
 -- (2.8,0) 
 -- (2.6,-1) 
-- (2.2,-1) 
 -- cycle;

\draw
    (3.5,0) 
 -- (3.8,0) node {\Large $\gamma - \frac{1}{2}$}		
 -- (4.1,0) 
 -- (4.3,-1) 
-- (4,-1) node {\Large $\gamma - \frac{1}{2}$}
-- (3.6,-1) 
 -- cycle;
\end{tikzpicture} 
\quad or  \quad
\begin{tikzpicture}
\draw
    (-1.5,0) 
 -- (-1.1,0) node {\Large $\tau_1+1$}
-- (-0.7,0) 
-- (-0.9,-1) 
-- (-1.3,-1) 
 -- cycle;

\draw
    (0,0) 
 -- (0.4,0) node {\Large $\tau_2$}
-- (0.8,0) 
-- (0.6,-1) 
-- (0.2,-1) 
 -- cycle;

\draw (1.5,-0.5) node {\Large $\dots$};

\draw
    (2,0) 
 -- (2.4,0) node {\Large $\tau_{p-q}$}		
 -- (2.8,0) 
 -- (2.6,-1) 
-- (2.2,-1) 
 -- cycle;

\draw
    (3,0) 
 -- (3.4,0) node {\Large $\gamma$}		
 -- (3.8,0) 
 -- (3.6,-1) 
-- (3.2,-1) 
 -- cycle;

\draw
    (4.2,0) 
 -- (4.6,0) 
 -- (4.8,-1) 
-- (4.4,-1) node {\Large $\gamma-1$}
-- (4,-1) 
 -- cycle;
\end{tikzpicture} 
\end{center}
Suppose on the contrary that the intertwining map \eqref{eq-upq1} has a kernel up to $\mathfrak{p}^+$.
Then the character formula of the irreducible quotient ($=$ image) of
\eqref{eq-upq1} must be an alternating sum of standard modules having infinitesimal character equal to \eqref{eq-infl1},
and one of the summands must be a standard module whose $\lambda_a$-datum is of one of the above forms.
However, it is obvious (from the inequality $(\gamma + \nu_x + t) - \tau_1 > 1$) that none of the above $\lambda_a$-data yield
representations with infinitesimal character \eqref{eq-infl1}.

Consequently, the character formula of the irreducible quotient in
\eqref{eq-upq1} contains no standard modules with lowest $K$-types at level $\mathfrak{p}^+$,
and hence the intertwining map \eqref{eq-upq1} has no kernel up to $\mathfrak{p}^+$.

\medskip
In conclusion, the intertwining operators corresponding to the Weyl group elements (i) - (iv) above have
no kernel on the level of $\mathfrak{p}^{\pm}$. Hence the claim holds, i.e. $I(t)$ and $J(t)$ has the same $K$-type multiplicities up to level $\mathfrak{p}^+$. As a result,  
the same arguments in the proof of Proposition \ref{prop-parallel} applies, and $\Sigma$ is non-unitary up to level $\mathfrak{p}^{+}$ by the first paragraph of the proof of the theorem.
\end{proof}

\begin{example} \label{eg-gammabig}
Let $G = U(5,4)$ and $\Sigma$ be the representation with combinatorial $\theta$-stable data:
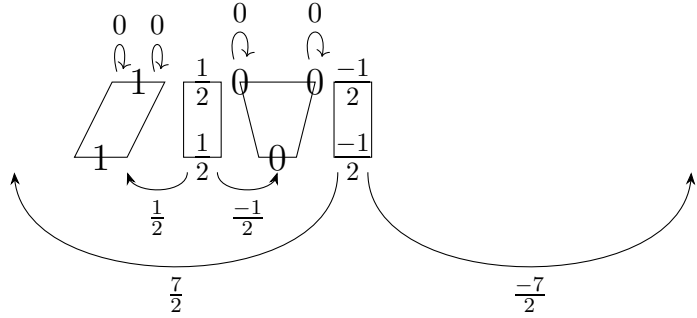
\begin{center}
\begin{tikzpicture}
\draw
    (0,0) node (1) {\Large $0$}
 -- (1,0) node (2) {\Large $0$} 
 -- (0.75,-1)
-- (0.5,-1) node {\Large $0$}
-- (0.25,-1) 
 -- cycle;

\draw[arrows = {-Stealth[]}]          (1.3,-1.2)   to [out=-90,in=-90]node[below]{$\frac{7}{2}$} (-3,-1.2);
\draw[arrows = {-Stealth[]}]          (1.7,-1.2)   to [out=-90,in=-90]node[below]{$\frac{-7}{2}$} (6,-1.2);

\draw
    (1.25,0) 
 -- (1.5,0) node {\Large $\frac{-1}{2}$}
-- (1.75,0)
-- (1.75,-1) 
-- (1.5,-1) node {\Large $\frac{-1}{2}$}
-- (1.25,-1)
 -- cycle;

\path (1) edge     [loop above]       node[] {$0$}         (1);
\path (2) edge     [loop above]       node[] {$0$}         (2);

\draw
    (-0.25,0) 
 -- (-0.5,0) node {\Large $\frac{1}{2}$}
-- (-0.75,0)
-- (-0.75,-1) 
-- (-0.5,-1) node {\Large $\frac{1}{2}$}
-- (-0.25,-1)
 -- cycle;
\draw[arrows = {-Stealth[]}]          (-0.7,-1.2)   to [out=-90,in=-90]node[below]{$\frac{1}{2}$} (-1.5,-1.2);
\draw[arrows = {-Stealth[]}]          (-0.3,-1.2)   to [out=-90,in=-90]node[below]{$\frac{-1}{2}$} (0.5,-1.2);

\draw
    (-1,0) 
 -- (-1.1,0) node (3) {} 
 -- (-1.35,0) node {\Large $1$}
-- (-1.6,0) node (4) {}
-- (-1.7,0) 
-- (-2.2,-1) 
-- (-1.85,-1) node {\Large $1$}
-- (-1.5,-1)
 -- cycle;

\path (3) edge     [loop above]       node[] {$0$}         (3);
\path (4) edge     [loop above]       node[] {$0$}         (4);
\end{tikzpicture}
\end{center}
The above blocks are semi-spherical (Definition \ref{def-semisph}) with infinitesimal character
$\Lambda = (3,1,1,1,0,0,0,0,-4)$ and lowest $K$-type $\delta$ of highest weight $\mu = (0,0,0,0,0 | 2,1,0,-1)$.
By Theorem \ref{thm-semipm}, the form of $\Sigma$ is indefinite on one of the level $\mathfrak{p}^+$ $K$-types
of highest weights:
$$(1,0,0,0,0 | 1,1,0,-1), \quad (1,0,0,0,0 | 2,0,0,-1),$$ 
$$(1,0,0,0,0 | 2,0,-1,-1), \quad (1,0,0,0,0 | 2,1,0,-2).$$
We implement this representation using \texttt{atlas}:
\begin{verbatim}
atlas> set G = U(5,4)
atlas> set p = parameter(G,7070,[3,1,1,1,0,0,0,0,-4]/1,[7,0,0,1,-1,0,0,0,-7]/2)
\end{verbatim}
To check it has the correct infinitesimal character $\Lambda$:
\begin{verbatim}
atlas> infinitesimal_character(p)
Value: [  3,  1,  1,  1,  0,  0,  0,  0, -4 ]/1
\end{verbatim}
More explicitly, by removing all non-zero values of $\nu$, one can check the $\lambda_a$-value of $\Sigma$
by the following:
\begin{verbatim}
atlas> infinitesimal_character(p*0)
Value: [ -1,  2,  2,  1,  1,  0,  0,  0, -1 ]/2
\end{verbatim}
The lowest $K$-type of $\Sigma$ is given by:
\begin{verbatim}
atlas> print_branch_irr_long(p,KGB(G,125),height(p))
sig  x    lambda                  hw                    dim  height
1    467  [1,1,1,0,0,0,0,0,-1]/1  [0,0,0,0,0,2,1,0,-1]  64   24
\end{verbatim}
Now we look at the signatures of some $K$-types of $X$:
\begin{verbatim}
atlas> print_sig_irr_long (p,KGB(G,125),height(p)+14)
sig  x     lambda                    hw                  dim  height
1    467   [1,1,1,0,0,0,0,0,-1]   [0,0,0,0,0,2,1,0,-1]   64   24
s    302   [1,1,1,0,0,0,0,0,-1]   [0,0,0,0,-1,2,1,0,0]   100  25
1    556   [2,1,0,0,0,0,0,0,-1]   [1,0,0,0,0,2,0,0,-1]   180  28
...
2    1493  [2,1,1,0,0,0,-1,0,-1]  [1,0,0,0,0,2,1,-1,-1]  300  34
...
3+s  610   [2,1,1,0,0,0,0,-1,-1]  [1,0,0,0,0,2,1,0,-2]   700  38
\end{verbatim}
The opposite signature occurs at the $K$-type 
with highest weight $(1,0,0,0,0 | 2,1,0,-2)$ as stated in Theorem \ref{thm-semipm}. And this is the only $K$-type
with indefinite forms up to level $\mathfrak{p}^+$.

Note also that $\Sigma$ also has opposite signature at level $\mathfrak{p}^-$ $K$-type $(0,0,0,0,-1|2,1,0,0)$. This is due to the $\lambda$-large block with content $\frac{-1}{2}$ also has indefinite form on the $\mathfrak{p}^-$-level $K$-type. 
\end{example}

\noindent{\it Proof of Theorem \ref{thm-upq} for Case (b).} Suppose the combinatorial $\theta$-stable data of $\Pi$ consists of a $\lambda$-large block of the form \eqref{eq-blocks}. There are a few possibilities for the blocks next to the $\lambda$-large $\gamma$-block. We only present the case when both neighboring blocks are not rectangular -- the other cases (i.e. there is one or two rectangular neighboring blocks) can be proved by similar arguments.

\medskip
There are four possibilities if both neighboring blocks are not rectangular:
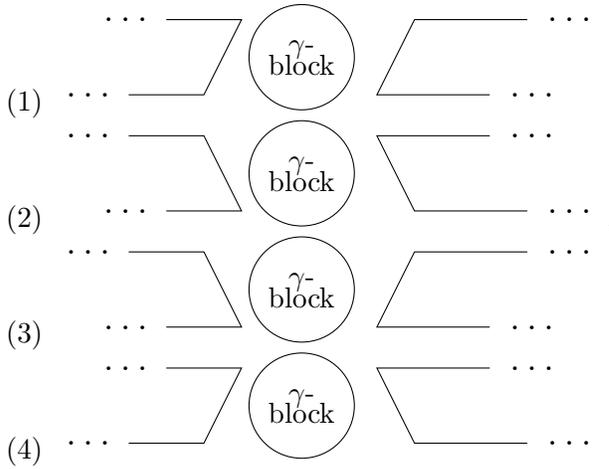
\begin{enumerate}
\item \begin{tikzpicture}
\draw
    (1,0) 
-- (2,0)
 -- (1.5,-1)  
-- (0.5,-1);

\draw (0.5,0) node {\Large $\dots$};
\draw (0,-1) node {\Large $\dots$};

\draw (2.8,-0.4) node {$\gamma$-};
\draw (2.8,-0.6) node {block};
\draw (2.8,-0.5) circle (0.7);

\draw
    (5.8,0) 
 -- (4.3,0) 
-- (3.8,-1)
 -- (5.3,-1);

\draw (6.4,0) node {\Large $\dots$};
\draw (5.9,-1) node {\Large $\dots$};

\end{tikzpicture}

\item \begin{tikzpicture}
\draw
    (0.5,0) 
-- (1.5,0)
 -- (2,-1)  
-- (1,-1);

\draw (0,0) node {\Large $\dots$};
\draw (0.5,-1) node {\Large $\dots$};

\draw (2.8,-0.4) node {$\gamma$-};
\draw (2.8,-0.6) node {block};
\draw (2.8,-0.5) circle (0.7);

\draw
    (5.3,0) 
 -- (3.8,0) 
-- (4.3,-1)
 -- (5.8,-1);

\draw (5.9,0) node {\Large $\dots$};
\draw (6.4,-1) node {\Large $\dots$};
\end{tikzpicture}. 
\item \begin{tikzpicture}
\draw
    (0.5,0) 
-- (1.5,0)
 -- (2,-1)  
-- (1,-1);

\draw (0,0) node {\Large $\dots$};
\draw (0.5,-1) node {\Large $\dots$};

\draw (2.8,-0.4) node {$\gamma$-};
\draw (2.8,-0.6) node {block};
\draw (2.8,-0.5) circle (0.7);

\draw
    (5.8,0) 
 -- (4.3,0) 
-- (3.8,-1)
 -- (5.3,-1);

\draw (6.4,0) node {\Large $\dots$};
\draw (5.9,-1) node {\Large $\dots$};

\end{tikzpicture}

\item \begin{tikzpicture}
\draw
    (1,0) 
-- (2,0)
 -- (1.5,-1)  
-- (0.5,-1);

\draw (0.5,0) node {\Large $\dots$};
\draw (0,-1) node {\Large $\dots$};

\draw (2.8,-0.4) node {$\gamma$-};
\draw (2.8,-0.6) node {block};
\draw (2.8,-0.5) circle (0.7);

\draw
    (5.3,0) 
 -- (3.8,0) 
-- (4.3,-1)
 -- (5.8,-1);

\draw (5.9,0) node {\Large $\dots$};
\draw (6.4,-1) node {\Large $\dots$};

\end{tikzpicture}

\end{enumerate}

In Cases (1) -- (2), consider the $U(r,s)$-module corresponding to the single $\lambda$-large $\gamma$-block. By hypothesis, the $\nu$-value of this $\gamma$-block must have a $> 1$ gap, so one can follow the proof of Theorem \ref{thm-semipm} that it is non-unitary on the level of $\mathfrak{p}^{\pm}$. By Proposition \ref{prop-bottom}, Case (1) is $\mathfrak{p}^-$-bottom layer, Case (2) is $\mathfrak{p}^+$-bottom layer. So the whole module is also non-unitary. 

In Case (3), if the $\gamma$-block of the form \eqref{eq-blocks} is trapezoidal, then the $\gamma$-block is both $\mathfrak{p}^{\pm}$-bottom layer; if the block is a parallelogram, then it is $\mathfrak{p}^-$-bottom layer. In both situations, the arguments are the same as in Case (1) -- (2). If the block is rectangular, i.e.
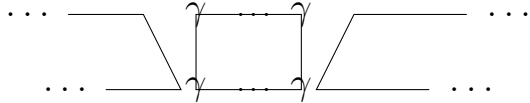
\begin{center}
\begin{tikzpicture}
\draw
    (0.5,0) 
-- (1.5,0)
 -- (2,-1)  
-- (1,-1);

\draw (0,0) node {\Large $\dots$};
\draw (0.5,-1) node {\Large $\dots$};

\draw 
 (2.2,0) node {\Large $\gamma$}
-- (3,0) node {$\dots$}
-- (3.6,0) node {\Large $\gamma$}
-- (3.6,-1) node {\Large $\gamma$}
-- (3,-1) node {$\dots$}
-- (2.2,-1) node {\Large $\gamma$}
-- cycle;

\draw
    (5.8,0) 
 -- (4.3,0) 
-- (3.8,-1)
 -- (5.3,-1);

\draw (6.4,0) node {\Large $\dots$};
\draw (5.9,-1) node {\Large $\dots$};
\end{tikzpicture}
\end{center}
then one can flip the diagram `upside down' (i.e. $U(p,q)$ to $U(q,p)$) to Case (4).

In Case (4), Theorem \ref{thm-semipm} implies that
the representation corresponding to the semi-spherical component with rightmost block being the $\gamma$-block 
is non-unitary up to $\mathfrak{p}^+$. By the definition of semi-spherical component, its leftmost block satisfies the top-left hypothesis of Proposition \ref{prop-bottom}(a), and the rightmost block (for instance, the rectangular $\gamma$-block above, with the whole diagram flipped upside down) satisfies the bottom-right hypothesis of Proposition \ref{prop-bottom}(a). So 
all such $K \cap L$-types are $\mathfrak{p}^+$ bottom layer (Remark \ref{rmk-blocks}), and hence the indefiniteness still persists on $\Pi$.\qed

\section{Proof of Salamanca-Riba-Vogan's conjecture} \label{sec-general}
In this section, we will prove a slightly stronger version of Conjecture \ref{conj-original} for $G = U(p,q)$: 

\begin{theorem} \label{thm-general}
Let $X$ be an irreducible $(\mathfrak{g}, K)$-
module with a unitarily small lowest $K$-type and infinitesimal character $\Lambda$. Then $\Pi$ is unitary implies that $\Lambda$ satisfies \eqref{eq-hull}. 
Otherwise, it is non-unitary at a unitarily small $K$-type up to level $\mathfrak{p}$.
\end{theorem}

Here is a special case of Theorem \ref{thm-general} for fundamental representations:
\begin{proposition} \label{prop-svfund}
    Let $G = U(p,q)$, and $\Pi$ be a fundamental representation (Definition \ref{def-fund}) with infinitesimal character $\Lambda$. Then all lowest $K$-types of $\Pi$ are unitarily small, and Theorem \ref{thm-general} holds for $\Pi$.
 \end{proposition}
\begin{proof}
The first statement follows immediately from the fact that $\lambda_u = P(\lambda_a - \rho(\mathfrak{g}))$, which is proved in \cite[Section 2]{SRV98}. As for the second statement, if $\Pi$ is unitary, then Theorem \ref{thm-upq} implies that 
$\langle \Lambda, \alpha^{\vee}\rangle \leq 1$ for all simple roots $\alpha$, so that $\Lambda$ must lie in the convex hull \eqref{eq-hull} centered at
\[\lambda_u(\delta) = (\overbrace{m, \dots, m}^{(p+q)\ terms}),\]
where $m$ is mean of the coordinates of $\Lambda$. Otherwise, Theorem \ref{thm-upq} implies that $\Pi$ is not unitary up to level $\mathfrak{p}$, and one can apply the proof of \cite[Proposition 7.18]{SRV98} to conclude that the non-unitary level $\mathfrak{p}$ $K$-type of $\Pi$ is also unitarily small.
\end{proof}

We now prove the general case of Theorem \ref{thm-general}:

\medskip
\noindent {\it Proof of Theorem \ref{thm-general}.}\  First of all, partition the combinatorial $\theta$-stable datum of $X$ so that each sub-datum is fundamental in the sense of Definition \ref{def-fund}, and the gap between two neighbouring sub-data is $> 1$, 
i.e. the $\lambda_a$-datum of $X$ is of the form:

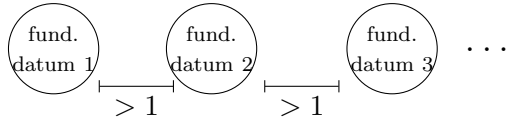
\begin{center}
\begin{tikzpicture}
\draw (-1.1,-0.3) node {\scriptsize{fund.}};
\draw (-1.1,-0.7) node {\scriptsize{datum 1}};
\draw (-1.1,-0.5) circle (0.6);


\draw [|-|] (-0.5,-1) -- node[below] {$>1$} (0.5,-1);

\draw (1,-0.3) node {\scriptsize{fund.}};
\draw (1,-0.7) node {\scriptsize{datum 2}};
\draw (1,-0.5) circle (0.6);


\draw   [|-|] (1.7,-1) -- node[below] {$>1$} (2.7,-1);

\draw (3.4,-0.3) node {\scriptsize{fund.}};
\draw (3.4,-0.7) node {\scriptsize{datum 3}};
\draw (3.4,-0.5) circle (0.6);

\draw (4.7,-0.5) node {\Large $\dots$};



\end{tikzpicture}
\end{center}
for instance, the $\lambda_a$-datum in Example \ref{eg-barbasch} is partitioned into two {\bf fundamental data}:

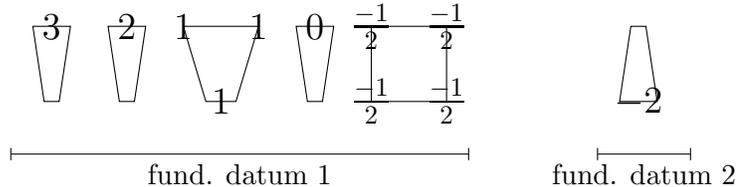
\begin{center}
\begin{tikzpicture}

\draw [|-|](-0.3,-1.7) -- node[below] {fund. datum 1} (5.8,-1.7);

\draw
    (0,0) 
 -- (0.25,0) node {\Large $3$}
-- (0.5,0) 
-- (0.35,-1) 
-- (0.15,-1) 
 -- cycle;

\draw
    (1,0) 
 -- (1.25,0) node {\Large $2$}
-- (1.5,0) 
-- (1.35,-1) 
-- (1.15,-1) 
 -- cycle;

\draw
    (2,0) node {\Large $1$}
 -- (3,0) node {\Large $1$}		
 -- (2.7,-1) 
-- (2.5,-1) node {\Large $1$}
-- (2.3,-1) 
 -- cycle;

\draw
    (3.5,0) 
 -- (3.75,0) node {\Large $0$}
-- (4,0) 
-- (3.85,-1) 
-- (3.65,-1) 
 -- cycle;

\draw
    (4.5,0) node {\Large $\frac{-1}{2}$}
 -- (5.5,0) node {\Large $\frac{-1}{2}$}
-- (5.5,-1) node {\Large $\frac{-1}{2}$}
-- (4.5,-1) node {\Large $\frac{-1}{2}$}
 -- cycle;

\draw [|-|](7.5,-1.7) -- node[below] {fund. datum 2} (8.75,-1.7);

\draw
    (7.95,0) 
 -- (8.15,0) 
-- (8.3,-1) 
-- (8.05,-1) node {\Large $-2$}
-- (7.8,-1) 
 -- cycle;
\end{tikzpicture} 
\end{center}

We proceed by induction on the number of fundamental $\theta$-stable data
of $X$. To begin with, suppose $X$ contains only one fundamental $\theta$-stable
datum, then $X$ is a fundamental representation, and the theorem is proved in Proposition \ref{prop-svfund}. 

By induction hypothesis, suppose the theorem holds for all $X_k$ having $k$ fundamental $\theta$-stable data. Now let $X_{k+1}$ be such that it has $(k+1)$ fundamental $\theta$-stable data. Separate the first $k$ $\theta$-stable data and the last $\theta$-stable datum of $X_{k+1}$, and let $Y$ (in $U(p_1,q_1)$), $Z$ (in $U(p_2,q_2)$) be the irreducible representations corresponding to these two $\theta$-stable data.

\medskip
We {\it claim} that the infinitesimal characters of $Y$ and $Z$ have to be in the convex hull \eqref{eq-hull}. Indeed, suppose on the contrary that one of them lies outside of \eqref{eq-hull}, then by induction hypothesis (on $Y$) or by Proposition \ref{prop-svfund} (on $Z$), it must be non-unitary up to level $\mathfrak{p}$. However, by considering the $\theta$-stable parabolic subalgebra $\mathfrak{q}_0 = \mathfrak{l}_0 + \mathfrak{u}_0$ with
\[\mathfrak{l}_0 = \mathfrak{l}_{0,1} \oplus \mathfrak{l}_{0,2} := \mathfrak{u}(p_1,q_1) \oplus \mathfrak{u}(p_2,q_2),\]
then Proposition \ref{prop-bottom} (or Remark \ref{rmk-blocks}) and the partitioning of $\theta$-stable data imply that all $(L \cap K)$-types of level $\mathfrak{p}$ are bottom layer. This implies that $X_{k+1}$ is also non-unitary up to level $\mathfrak{p}$, which gives a contradiction.

\medskip
Since $X_{k+1}$ is a lowest $K$-type subquotient of $\mathcal{R}_{\mathfrak{q}}(Y \boxtimes Z)$, the arguments in the previous paragraph along with the characterization of convex hull given in \cite[Proposition 1.10]{SRV98} implies that the infinitesimal character of $X_{k+1}$ is of the form
\begin{equation} \label{eq-lambdayz}
    \Lambda = (\overbrace{m_1, \cdots, m_1}^{r_1\ terms};\overbrace{m_2, \cdots, m_2}^{r_2\ terms}) + \sum_{\alpha \in \Delta^+(\mathfrak{l},\mathfrak{h})} c_{\alpha}\alpha, \quad \quad |c_{\alpha}| \leq \frac{1}{2}, 
\end{equation}
where $r_1 := p_1+q_1$, $r_2 := p_2+q_2$, and $m_1$ (resp. $m_2$) is the mean of coordinates of the first $k$ (resp. last) fundamental $\theta$-stable datum.
Using the notations in \eqref{eq-lambdayz}, one needs to prove that
$\Lambda$ lies in the convex hull \eqref{eq-hull} centered at 
\begin{equation} \label{eq-m}
\lambda_u(\delta) = (\overbrace{m, \cdots, m}^{r_1+r_2\ terms}), \quad \quad \quad where\ \ m = \frac{r_2m_1 + r_1m_2}{r_1+r_2}.\end{equation}

To see so, first note that the unitarily small hypothesis on $X_{k+1}$ implies that $m_1 \geq m_2$ in \eqref{eq-lambdayz} cannot be `far apart'. More explicitly, by looking at the formula of $\Lambda = (\lambda_a(\delta), \nu)$ in \eqref{eq-lambdayz}, one can conclude for $\lambda_a(\delta)$, the mean of its first $r_1$ (resp. last $r_2$) coordinates equals $m_1$ (resp. $m_2$). 
Therefore, the mean of the first $r_1$ (resp. last $r_2$) coordinates of $\lambda_a(\delta) - \rho(\mathfrak{g})$ is $m_1 - \frac{r_2}{2}$ (resp. $m_2 + \frac{r_1}{2}$).

Suppose on contrary that $m_1 - m_2 > \frac{r_1+r_2}{2}$. Then 
$\lambda_a(\delta) - \rho(\mathfrak{g})$
is closer to 
the point 
\[(\overbrace{m_1 - \frac{r_2}{2}, \cdots, m_1 - \frac{r_2}{2}}^{r_1\ terms};\overbrace{m_2 + \frac{r_1}{2}, \cdots, m_2 + \frac{r_1}{2}}^{r_2\ terms})\] 
in the dominant Weyl chamber than $(\overbrace{m, \cdots, m}^{r_1+r_2\ terms})$. In other words, $\lambda_u(\delta) = P(\lambda_a(\delta)-\rho(\mathfrak{g}))$ is not equal to $(\overbrace{m, \cdots, m}^{r_1+r_2\ terms})$, contradicting the fact that the lowest $K$-type of $X_{k+1}$ is unitarily small.
Therefore, one must have
\begin{equation} \label{eq-m1m2}
    0 \leq m_1 - m_2 \leq \frac{r_1+r_2}{2}.
\end{equation}

Combining Equations \eqref{eq-lambdayz}, \eqref{eq-m} and \eqref{eq-m1m2}, one can show that $\Lambda$ lies in the convex hull \eqref{eq-hull}. More explicitly,  
\begin{align*}
\Lambda - \lambda_u(\delta) &= (m_1-m, \cdots, m_1-m; m_2 - m, \cdots, m_2 - m) + \sum_{\alpha \in \Delta^+(\mathfrak{l},\mathfrak{h})} c_{\alpha}\alpha \\
&= \frac{m_1-m_2}{r_1+r_2}(r_2, \cdots, r_2; -r_1, \cdots, -r_1) + \sum_{\alpha \in \Delta^+(\mathfrak{l},\mathfrak{h})} c_{\alpha}\alpha \\
&= \frac{m_1-m_2}{r_1+r_2} \sum_{\beta \in \Delta^+(\mathfrak{u},\mathfrak{h})} \beta + \sum_{\alpha \in \Delta^+(\mathfrak{l},\mathfrak{h})} c_{\alpha}\alpha
\end{align*}
which is a sum of positive roots in $\Delta^+(\mathfrak{g},\mathfrak{h})$, whose coefficients having absolute values $\leq \frac{1}{2}$. Therefore, the result follows from \cite[Proposition 1.10]{SRV98} again. \qed

\section{Proof of Vogan's FPP Conjecture} \label{sec-fpp}
We now proceed to proving Conjecture \ref{conj-fpp} for $G=U(p,q)$ by applying the results in Section \ref{sec-fund}.
To do so, one needs to understand the good range condition (Definition \ref{def-good}) using combinatorial $\theta$-stable data.

\medskip
Let $X$ be an irreducible, Hermitian $(\mathfrak{g},K)$-module with real infinitesimal character and a lowest $K$-type $\delta$.
All $\theta$-stable parabolic subalgebras $\mathfrak{q}'$ containing the quasisplit parabolic subalgebra $\mathfrak{q}$ defined by $\lambda_a(\delta)$
are in $1-1$ correspondence with the set of partitions of the 
$\lambda_a$-datum in the combinatorial $\theta$-stable datum of $X$. For instance, the smallest possible Levi subalgebra $\mathfrak{q}$ corresponds to the finest partition of the $\lambda_a$-datum, each consisting of a single $\gamma$-block. On the other extreme, $\mathfrak{g}$ corresponds to the single partition containing all $\gamma$-blocks.

As a consequence, all such $\theta$-stable parabolic subalgebras $\mathfrak{q}' \supset \mathfrak{q}$ corresponds to a partition 
\[\mathcal{D} = \bigsqcup_{i=1}^k \mathcal{D}_i\]
of the $\theta$-stable datum of $X$, and the infinitesimal character $\Lambda = (\lambda_a(\delta),\nu)$ of $X$ is also partitioned into
\[(\Lambda_1, \cdots, \Lambda_k)\]
up to permutation of coordinates. Define the $i^{th}$-{\bf segment} of the partition by the line segment $[e_i, b_i]$, where $e_i$ (resp. $b_i$) is the largest (resp. smallest) number
in $\Lambda_i$. 

\begin{example} \label{eg-good1}
    Let $G = U(5,4)$, and $\Sigma$ be the irreducible representation given in Example \ref{eg-gammabig}. Suppose we partition the combinatorial $\theta$-stable datum of $\Sigma$ by $\mathcal{D} = \mathcal{D}_1 \sqcup \mathcal{D}_2$, where
    \[\mathcal{D}_1 := \{1\text{-block},\ (\frac{1}{2})\text{-block}\} \quad \quad \quad \mathcal{D}_2 := \{0\text{-block},\ (\frac{-1}{2})\text{-block}\}\]
Then $\Lambda_1 = (1,1,\frac{1}{2}+\frac{1}{2},\frac{1}{2}-\frac{1}{2}) = (1,1,1,0)$ and 
    $\Lambda_2 = (0,0+0,0+0,\frac{-1}{2}+\frac{7}{2},\frac{-1}{2}-\frac{7}{2}) = (0,0,0,3,-4)$, and hence
    \[[e_1,b_1]=[1,0], \quad \quad \quad [e_2, b_2] = [3,-4].\]
\end{example}

The following lemma determines whether $X$
is cohomologically induced from some proper $\theta$-stable parabolic subalgebra in good range by looking at its combinatorial $\theta$-stable data:

\begin{lemma} \label{lem-good}
    Retain the setting in the above paragraphs. 
    Then $X$
    is cohomologically induced from a $\theta$-stable parabolic subalgebra $\mathfrak{q}'$ if and only if the segments of the partition of the combinatorial $\theta$-stable datum of $X$ corresponding to $\mathfrak{q}'$ satisfy the following inequalities:
    \[e_1 \geq b_1 > e_2 \geq b_2 > \cdots > e_k \geq b_k.\]
\end{lemma}
The lemma follows immediately from Definition \ref{def-good}. For instance, the above lemma
implies that the module $\Sigma$ in Example \ref{eg-good1} cannot be cohomologically induced from $\mathfrak{q}' = \mathfrak{l}'+\mathfrak{u}'$ with $\mathfrak{l}'_0 = \mathfrak{u}(2,2) \oplus \mathfrak{u}(3,2)$ in good range.
Indeed, one can easily check that $\Sigma$ cannot be cohomologically induced from any proper parabolic subalgebra $\mathfrak{q}' \supset \mathfrak{q}$ of $\mathfrak{g}$. In other words, $\Sigma$ is \emph{fully supported}.

\bigskip
\noindent {\it Proof of Conjecture \ref{conj-fpp} for  $U(p,q)$.} As in the proof of Conjecture \ref{conj-original} in Section \ref{sec-general}, partition the combinatorial $\theta$-stable datum of $X$ into fundamental data $\mathcal{D} = \bigsqcup_{i=1}^k \mathcal{D}_i$. Note that the infinitesimal character $\Lambda_i$ satisfies Conjecture \ref{conj-fpp} for all $i$, otherwise the same arguments as in Section \ref{sec-general} implies that $X$ is not unitary up to level $\mathfrak{p}$. In other words, the coordinates of $\Lambda_i$ in the $i^{th}$ segment $[e_i,b_i]$, rearranged in descending order, must have gaps $\leq 1$. 

By Lemma \ref{lem-good} and the fact that $X$ is not cohomologically induced from any proper $\theta$-stable parabolic subalgebras, the segments $\{[e_i,b_i]\ |\ i = 1, \dots, k\}$ must be \emph{interlaced}, i.e. for any $1 \leq a < b \leq k$ , there
exist $1 \leq$ $i_0 = a$, $i_1$, $\dots$, $i_m = b$ $\leq k$ such that 
\[[e_{i_n}, b_{i_n}] \cap [e_{i_{n+1}}, b_{i_{n+1}}] \neq \phi\] 
for all $n$ (the proof is similar to that of \cite[Proposition 3.1]{DW20}). Combined with the fact that $\Lambda \sim (\Lambda_1, \dots, \Lambda_k)$, where the coordinates of each $\Lambda_i \in [e_i, b_i]$ have gaps $\leq 1$, this implies that the gaps between consecutive coordinates of $\Lambda$ (in descending order) also have gaps $\leq 1$, i.e. $\langle \Lambda, \alpha^{\vee} \rangle \leq 1$ for all simple roots $\alpha$. \qed

\begin{remark}
As a final remark, we note that the proof of Salamanca-Riba-Vogan's conjecture and Vogan's FPP conjecture are based on studies of fundamental representations in Section \ref{sec-fund}. We believe that these representations play a pivotal role in the study of the full unitary dual of $U(p,q)$.
\end{remark}


\end{document}